\documentclass[11pt,reqno]{amsart}
\usepackage{amsmath}
\usepackage{amssymb}
\usepackage{cite}
\usepackage{color}

\theoremstyle{plain}
\newtheorem{theorem}{\bf Theorem}[section]
\newtheorem{proposition}[theorem]{\bf Proposition}
\newtheorem{lemma}[theorem]{\bf Lemma}
\newtheorem{corollary}[theorem]{\bf Corollary}

\theoremstyle{definition}

\newtheorem{remark}[theorem]{\bf Remark}

\numberwithin{equation}{section}

\begin{document}
\title[solutions for fractional Nirenberg problem]
{Multi-bump solutions for fractional Nirenberg problem}

\address[Chungen Liu]{School of Mathematical Sciences and LPMC, Nankai University, Tianjin
300071, P.R. China} \email{liucg@nankai.edu.cn}\thanks{$^1$The first author is partially supported by the NSF of China (11471170).}

\address[Qiang Ren]{School of Mathematical Sciences and LPMC, Nankai University, Tianjin
300071, P.R. China} \email{tjftp@mail.nankai.edu.cn}

\author{Chungen Liu$^1$
 and Qiang Ren}

\date{Completed: \today}



\begin{abstract}
We consider  the multi-bump solutions of the following fractional Nirenberg problem
\begin{equation}\label{01}
(-\Delta)^s u=K(x)u^{\frac{n+2s}{n-2s}}, \;\;\;\;u>0\;\;\text{ in }\mathbb{R}^n,
\end{equation}
where $s\in (0,1)$ and $n>2+2s$. If $K$ is a periodic function in some $k$ variables with $1\leq k<\frac{n-2s}2$, we proved that \eqref{01} has multi-bump solutions with bumps clustered on some lattice points in $\mathbb{R}^k$ via Lyapunov-Schmidt reduction. It is also established that the equation \eqref{01} has an infinite-many-bump solutions with bumps clustered on some lattice points in $\mathbb{R}^n$ which is isomorphic to $\mathbb{Z}_+^k$.
\end{abstract}
\maketitle

\section{Introduction and main results}
The classic Nirenberg problem asks that on the standard sphere $(\mathbb{S}^n,g_{\mathbb{S}^n})$ with $n\geq 2$, whether there exists a function $w$ such that the scalar curvature( Gauss curvature in the dimension 2) of the conformal metric $g=e^wg_{\mathbb{S}^n}$ equals to a prescribed function $\tilde K$. This probelm is equivalent to solving the following equations
\begin{equation}\label{31}
-\Delta_{g_{\mathbb{S}^n}}w+1=\tilde K e^{2w}\;\;\text{on}\;\;\mathbb{S}^2
\end{equation}
and
\begin{equation}\label{32}
-\Delta_{g_{\mathbb{S}^n}}v+\frac{n-2}{4(n-1)}R_{g_{\mathbb{S}^n}}v=\frac{n-2}{4(n-1)}\tilde Kv^{\frac{n+2}{n-2}}\;\;\text{on}\;\; \mathbb{S}^n\;\; \text{for}\;\; n\geq 3,
\end{equation}
where $R_{g_{\mathbb{S}^n}}=n(n-1)$ is the scalar curvature of $(\mathbb{S}^n,g_{ \mathbb{S}^n})$ and $v=e^{\frac{n-2}4w}$.
The linear operators defined on left-hand side of the equation \eqref{31} and \eqref{32} are called the conformal Laplacian on $\mathbb{S}^n$.

For any Riemannian manifold $(M,g)$, the conformal Laplacian is defined by $P_1^g=-\Delta_g+\frac{n-2}{4(n-1)}R_g$, where $R_g$ is the scalar curvature of $(M,g)$. Let $u>0$ and $h=u^{\frac{4}{n-2}}g$, the conformal Laplacian has the following conformally invariant property
\[P_1^g(u\phi)=u^{\frac{n+2}{n-2}}P_1^h(\phi)\;\;\;\;\text{for}\;\;\phi\in C^{\infty}(M).\]
The Paneitz operator $P_2^g$ is another interesting conformal invariant operator. It was defined   in \cite{Paneitz} by
\[P_2^g=\Delta^2_g+div_g(a_n R_gId-b_n \mathcal{R}ic_g)\nabla_g+\frac{n-4}2Q_g,\]
where $a_n=\frac{(n-2)^2+4}{2(n-1)(n-2)}$, $b_n=-\frac4{n-2}$, $\mathcal{R}ic:TM \to TM $ is a (1,1)-tensor operator defined by $\mathcal{R}ic_i^j=g^{jk}Ric_{ki}$ and $Q_g=-\frac2{(n-2)^2}|Ric_g|^2+\frac{n^3-4n^2+16n-16}{8(n-1)^2(n-2)^2}R_g^2-\frac1{2(n-1)}\Delta_g R_g$ which is called the $Q$-curvature of $(M,g)$.

Later on, more conformally covariant elliptic operators were found.  The operator $P_1^g$ and $P_2^g$ were generalized by Graham, Jenne, Mason and Sparling in \cite{Graham1992} to a sequence of integer order conformally covariant elliptic operators $P_k^g$ for $k\in \mathbb{N}_+$ if $n$ is odd; and $k\in\{1,\dots \frac n2\}$ if $n$ is even. Furthermore,  any real number order conformally covariant pseudo-differential operator was intrinsically defined by Peterson in \cite{Peterson2000}. Graham and Zworski in \cite{Graham2003} proved that the operators $P_k^g$ can be considered as the residue of a meromorphic family of scattering operators $S(s)$ at $s=\frac n2+k$. Then a family of non-integer order conformally covariant pseudo-differential operators $P_s^g$( $0<s<\frac n2$) were naturally defined. Using the localization method in \cite{Caffarelli2007}, Chang and Gonz\'{a}lez \cite{Chang2012} showed that for any $s\in(0,\frac n2)$, the operator $P_s^g$ can also be defined as a Dirichlet-to-Neumann operator of a conformally compact Einstein manifold.

The conformally covariant law for $P_s^g$ means that for any Riemannian manifold $(M,g)$ and a conformal transformation $h=v^{\frac4{n-2s}}g$, $v>0$, there  holds
\begin{equation*}
P^g_s(v\phi)=v^{\frac{n+2s}{n-2s}}P^h_s(\phi)\;\;\;\;\text{for}\;\;\phi\in C^{\infty}(M).
\end{equation*}
Especially, $P^g_s(1)$ is called the $Q_s$ curvature or $s$-curvature of $(M,g)$( see \cite{Chang2012} and \cite{JinTianling2011} for example).

The fractional Nirenberg problem was naturally raised on $Q_s$ curvature, it asks that on the standard sphere $\mathbb{S}^n$, whether there exists a function $v>0$ such that the $Q_s$ curvature of the conformal metric $g=v^{\frac{4}{n-s}}g_{\mathbb{S}^n}$ equals to a prescribed function $\tilde K$. It can be reduced to the existence of the solution of  the following equation
\begin{equation}\label{30}
P_s^{g_{\mathbb{S}^n}}(v)=\tilde Kv^{\frac{n+2s}{n-2s}},\;\; v>0\;\;\text{on}\;\; \mathbb{S}^n,
\end{equation}
where $s\in (0,1)$, $n>2s$ and $\tilde K$ is a given positive function.

It was shown in \cite{Branson1995} that the operator $P_{s}^{g_{\mathbb{S}^n}}$ is an intertwining operator and can be expressed as
\[P_s^{g_{\mathbb{S}^n}}=\frac{\Gamma(B+\frac12+s)}{\Gamma(B+\frac12-s)},\;\;\;\;B= \sqrt{-\Delta_{\mathbb{S}^n}+ \left(\frac{n-1}2\right)^2},\]
where $\Delta_{\mathbb{S}^n}$ is the Beltrami-Laplacian operator. What is more, $P_{s}^{g_{\mathbb{S}^n}}$ is more concrete under the stereographic projection. Let
\[F:\mathbb{R}^n\to \mathbb{S}^n\char92\{\mathcal{N}\},\;\;\;\;x\mapsto\left(\frac{2x}{|x|^2+1},\frac{|x|^2-1}{|x|^2+1}\right)\]
be the inverse of stereographic projection, where $\mathcal{N}$ is the north pole of $\mathbb{S}^n$. Then it holds that
\[P_s^{g_{\mathbb{S}^n}}(\phi)\circ F=|J_F|^{-\frac{n+2s}{2n}}(-\Delta)^s(|J_F|^{\frac{n-2s}{2n}}(\phi\circ F)),\]
where $(-\Delta)^s$ is the fractional Laplacian defined by
\[(-\Delta)^s \phi(x)=C(n,s)P.V.\int_{\mathbb{R}^n}\frac{\phi(x)-\phi(y)}{|x-y|^{n+2s}}dy,\;\;\;\;\text{where}\;\;\phi\in C^{\infty}(\mathbb{R}^n).\]

If we write $u=|J_F|^{\frac{n-2s}{2n}}(v\circ F)$ and $K=\tilde K\circ F$, the equation \eqref{30} is transformed into
\begin{equation}\label{1}
(-\Delta)^s u=K(x)u^{\frac{n+2s}{n-2s}},\;\; u>0\;\;\text{in}\;\; \mathbb{R}^n,
\end{equation}
where $s\in(0,1)$ and $n>2s$. The existence of the solutions to the problem \eqref{30} has been proved under various conditions(see for exmaple \cite{Abdelhedi, Chen2014, ChenYH2015, Chen2016, JinTianling2011, JinTL2015, Abdelhedi2016, Chtioui2016}). The compactness of the solutions to \eqref{30} was studied in \cite{JinTianling2011}. Chen and Zheng \cite{ChenYH2015} found a 2-peak solution when $K(x)=1+\varepsilon \tilde{K}(x)$ has at least two critical points and satisfies some local conditions. What is more, Liu in \cite{Liu2016} constructed infinitely many 2-peak solutions when $K$ has a sequence of strictly local maximum points moving to infinity. When $K$ is a radial symmetric function, in \cite{LiuCG2016} and \cite{Long2016} it was showed that \eqref{1} has infinitely many non-radial solutions.

In this paper, we continue to study the bump solutions or peak solutions of \eqref{1}. Assume that $K$ satisfies the following conditions

\noindent$(H_1)$ $0<\inf_{\mathbb{R}^n} K\leq \sup_{\mathbb{R}^n} K<\infty$;

\noindent$(H_2)$ $K(x)$ is a $C^{1,1}$ function, and  $1$-periodic in the first $k$ variables $x_1, \cdots, x_k$;

\noindent$(H_3)$ $0$ is a critical point of $K$, and in a neighborhood of $0$, there is a number $\beta\in (n-2s,n)$  such that
\begin{equation*}
K(x)=K(0)+\sum_{i=1}^na_i|x_i|^\beta+R(x),
\end{equation*}
where $a_i\not=0$ for $i=1,\dots,n$, $\sum_{i=1}^n a_i<0$, $R(y)\in C^{[\beta]-1,1}$ and $\sum_{j=0}^{[\beta]}|\nabla^j R(y)||y|^{-\beta+j}=o(1)$ as $y\to 0$. Here $\nabla^j R(y)$ denote all of the possible derivatives of $R(y)$ of the order $j$.

We note that the conditions $(H_1)$-$(H_3)$ and the condition $(\mathcal{K})$ in \cite{Liu2016} has some intersecion. When $K$ satisfies both the condition $(\mathcal{K})$ in \cite{Liu2016} and  our conditions $(H_1)$-$(H_3)$, the equation \eqref{1} has infinitely many 2-peak solutions according to \cite{Liu2016}.

In this paper, we will show that equation \eqref{1} has solutions with large number bumps and its bumps located near some lattice points in $\mathbb{R}^k$ with $1\leq k<\frac{n-2s}2$.

Let $Q_m:=([0,m+1)^k\times \mathbf{0})\cap \mathbb{Z}^n$, where $m\in\mathbb{N}_+\cup \{\infty\}$ and $\mathbf{0}$ is a zero vector in $\mathbb{R}^{n-k}$.
\begin{theorem}\label{th1}
Suppose $n>2s+2$, $m\in \mathbb{N}_+\cup\{\infty\}$ and $1\leq k< \frac{n-2s}2$. If $K$ satisfies the conditions $(H_1)$-$(H_3)$, there exists an integer $l_0\in \mathbb{N}$, such that for any integer $l>l_0$, the equation \eqref{1} has a solution with its bumps clustered on $lQ_m$.
\end{theorem}
Notice that $Q_\infty$ is an infinite lattice which isomorphic to $\mathbb{Z}_+^k$. So we get an infinite-many-bump solution of the equation \eqref{2} via Theorem \ref{th1}.

In order to prove Theorem \ref{th1}, we assume $K(0)=1$ with no loss of generality. For any positive integer $l$, define $\lambda=l^{\frac{n-2s}{\beta-n+2s}}$. Then we have $\lambda^\beta=(\lambda l)^{n-2s}$. Using the transformation $u(x)\mapsto \lambda^{-\frac{n-2s}2}u(\frac{x}{\lambda})$, we can change the equation \eqref{1} into
\begin{equation}\label{2}
(-\Delta)^s u=K(\frac{x}{\lambda})u^{\frac{n+2s}{n-2s}},\;\; u>0,\;\;\text{in}\;\; \mathbb{R}^n.
\end{equation}
 The functional corresponding to equation \eqref{2} is
\begin{equation*}
I(u)=\frac12\int_{\mathbb{R}^n}|(-\Delta)^{\frac s2}u|^2-\frac{n-2s}{2n}\int_{\mathbb{R}^n}K\left(\frac{x}{\lambda}\right)(u_+)^{\frac{2n}{n-2s}},\;\;\;\; u\in \dot{H}^s(\mathbb{R}^n),
\end{equation*}
where $u_+=\max\{u,0\}$. The Hilbert space $\dot{H}^s(\mathbb{R}^n)$ is the completion of $C_0^\infty(\mathbb{R}^n)$ under the Gagliardao semi-norm (cf. \cite{diNezza2012} for detail)
\[[u]_{\dot{H}^s(\mathbb{R}^n)}:=\left(\int_{\mathbb{R}^n}\int_{\mathbb{R}^n}\frac{|u(x)-u(y)|^2}{|x-y|^{n+2s}} \right)^{\frac12}=\left(2C(n,s)\int_{\mathbb{R}^n}|(-\Delta)^{\frac s2}|^2\right)^{\frac12},\]
 where $C(n,s)$ is a constant depending on $n$ and $s$. It is well known that $\dot{H}^s(\mathbb{R}^n)$ can be imbedded into $L^{2^*(s)}( \mathbb{R}^n)$ and  the following Hardy-Littlewood-Sobolev inequality holds
\begin{equation}\label{3}
S\left(\int_{\mathbb{R}^n}|u|^{2^*(s)}\right)^{\frac {2}{2^*(s)}}\leq \int_{\mathbb{R}^n}|(-\Delta)^{\frac s2}u|^2,\;\;\;\;\text{for}\;\;u\in C^{\infty}_0(\mathbb{R}^n) ,
\end{equation}
where $s\in (0,1)$, $n>2s$ and $2^*(s)=\frac {2n}{n-2s}$. Lieb \cite{Lieb1983} proved that the extremals corresponding to the best constant $S$ of \eqref{3} are of the form
\[U_{\xi,\Lambda,C_0}=C_0\left(\frac{\Lambda}{1+\Lambda^2|x-\xi|^2}\right)^{\frac {n-2s}2},\]
where $C_0,\Lambda\in \mathbb{R}_+$ and $\xi\in\mathbb{R}^n$.

Choosing a suitable constant $C_0=C_0(n,s)$, we see that the function $U_{\xi,\Lambda}:=U_{\xi,\Lambda,C_0}$ solves the equation
\begin{equation}\label{4}
(-\Delta)^s u=u^{\frac{n+2s}{n-2s}},\;\;\;\; u>0 \text{ in }\mathbb{R}^n.
\end{equation}
Under some decay assumptions, \cite{Chen2006,LiYY2004, LiYY1995} proved that all the solutions of \eqref{4} are only of the form $U_{\xi,\Lambda}$. Furthermore, it was proved in \cite{Davila2013} that the solution $U_{\xi,\Lambda}$ of the equation \eqref{4} is nondegenerate, \textit{i.e.} any bounded solution of the equation $(-\Delta)^s \phi=\frac{n+2s}{n-2s}U_{\xi,\Lambda}^{\frac{4s}{n-2s}}\phi$ is a linear combination of $\frac{\partial U_{\xi,\Lambda}}{\partial \Lambda}$ and $\frac{\partial U_{\xi,\Lambda}}{\partial \xi_i}$, $i=1,2\dots,n$.

We will use the functions $U_{\xi,\Lambda}$ to construct the approximate solutions of the equation \eqref{2}. We define $X_{l,m}=\{\lambda l x|x\in Q_m\}$ and  arrange it in any way as a sequence $X_{l,m}=\{X^i\}_{i=1}^{(m+1)^k}$. Let $P^i\in B_{\frac12}(X^i)=\{X\in \mathbb{R}^n|\;|X-X^i|<\frac 12\}$, $\Lambda_i\in[C_1,C_2]$, for $i=1,2\dots, (m+1)^k$, where $C_1$ and $C_2$ are some positive numbers to be defined later( see \eqref{63}). Let
\[W_{m}(x):=\sum_{i=1}^{(m+1)^k}U_{P^i,\Lambda_i}(x)\]
to be an approximate solution of the problem \eqref{2}.
\begin{theorem}\label{th2}
Under the same conditions of Theorem \ref{th1}, there exists an interger $l_0>0$, such that for any integer $l>l_0$, equation \eqref{2} has a $C^2_{loc}$ solution $u_m$ of the form
\begin{equation*}
u_m=W_m+\phi_m,
\end{equation*}
where $m\in\mathbb{N}_+\cup\{\infty\}$, $|\phi|_{L^\infty(\mathbb{R}^n)}\to 0$ and $\displaystyle\max_{i=1,\dots,(m+1)^k}\{|P^i-X^i|\}\to 0$ as $l\to\infty$.
\end{theorem}

As a consequence of Theorem \ref{th2}, we have
\begin{corollary}\label{coro1}
Under the same conditions of Theorem \ref{th1}, the equation \eqref{1} has infinitely many multi-bump solutions.
\end{corollary}

Theorem \ref{th1} follows from Theorem \ref{th2}. So we only need to prove Theorem \ref{th2}. In this article we use $l$ as the pertubation parameter and follow the methods developed in \cite{LiYY2016,WeiYan2010}. In the section \ref{sect2}, we carry out the Liapunov-Schmidt reduction. Theorem \ref{th2} is proved in the section \ref{sect3}. Some useful estimations are  presented in Appendix \ref{sect4}. The expansions of the functional $\frac \partial {\partial \Lambda_i}I(W_m)$ and $\frac \partial {\partial P^i_{ j}}I(W_m)$ are shown  in Appendix \ref{sect5}.

In this article, $C$ denotes a varying constant independent of $m$.
\section{Finite dimensional reduction}\label{sect2}
In this section, we will carry out the Lyapunov-Schmidt reduction in the case of $m<\infty$.

We define two weighted norms
\[\|u\|_{*}=\sup_{y\in \mathbb{R}^n}\left(\gamma(y)\sum_{i=1}^{(m+1)^k}\frac1{(1+|y-X^i|)^{\frac{n-2s}2+\tau}}\right)^{-1}|u(y)|,\]
and
\[\|u\|_{**}=\sup_{y\in \mathbb{R}^n}\left(\gamma(y)\sum_{i=1}^{(m+1)^k}\frac1{(1+|y-X^i|)^{\frac{n+2s}2+\tau}}\right)^{-1}|u(y)|,\]
where
\[\gamma(y)=\min\left\{\min_{i=1,\dots,(m+1)^k}\left(\frac{1+|y-X^i|}\lambda\right)^{\tau-s},1\right\},\]
and $\tau\in ( k,\frac{n-2s}2)$ is a constant.

Consider the following equation
\begin{equation}\label{16}
(-\Delta)^s\phi-\frac{n+2s}{n-2s}K\left(\frac x \lambda\right)W_m^{\frac{4s}{n-2s}}\phi=g\;\;\text{in} \;\;\mathbb{R}^n.
\end{equation}

\begin{lemma}\label{lm6}
Let $\phi$ be a solution of the equation \eqref{16}, then we have the following estimate
\begin{eqnarray}\label{64}
\nonumber&&\left(\gamma(y)\sum_{h=1}^{(m+1)^k}\frac1{(1+|y-X^h|)^{\frac{n-2s}2+\tau}}\right)^{-1}|\phi(y)| \\
&\leq& C\|g\|_{**}+C\|\phi\|_*\left(\frac1{(\lambda l)^{\frac{4s}{n-2s}k}}\frac{\sum_{h=1}^{(m+1)^k}\frac1{(1+|y-X^h|)^{\frac{n-2s}2+\tau+\theta}}}{ \sum_{h=1}^{(m+1)^k}\frac1{(1+|y-X^h|)^{ \frac{n-2s}2+\tau}}}\right),
\end{eqnarray}
where $\theta>0$ is a constant and $C$ is independent of $m$.
\end{lemma}

\begin{proof}
We rewrite the equation \eqref{16} into an integral equation
\begin{equation}\label{33}
\phi(y)=C_1(n,s)\int_{\mathbb{R}^n}\frac1{|y-z|^{n-2s}}\left(\frac{n+2s}{n-2s}K(\frac z \lambda)W_m^{\frac{4s}{n-2s}}(z)\phi(z)+g(z)\right)dz,
\end{equation}
where the constant $C_1(n,s)$ is defined in the Green function of $(-\Delta)^s$ on $\mathbb{R}^n$(cf. \cite{Caffarelli2007}).

From Lemma \ref{lm11}, we get
\begin{eqnarray}\label{65}
\nonumber&&\left|\int_{\mathbb{R}^n}\frac1{|y-z|^{n-2s}}K(\frac z \lambda)W_m^{\frac{4s}{n-2s}}(z)\phi(z)dz\right| \\ \nonumber
&\leq& C\|\phi\|_*\int_{\mathbb{R}^n}\frac1{|y-z|^{n-2s}}W_m^{\frac{4s}{n-2s}}(z)\gamma(z)\sum_{h=1}^{(m+1)^k}\frac1{( 1+|z-X^h|)^{\frac{n-2s}2+\tau}}dz \\ \nonumber
&\leq&C\|\phi\|_*\left(\gamma(y)\sum_{h=1}^{(m+1)^k}\frac1{( 1+|y-X^h|)^{\frac{n-2s}2+\tau+\theta}}\right. \\
&&\left.+\frac1{(\lambda l)^{\frac{4s}{n-2s}k}}\gamma(y)\sum_{h=1}^{(m+1)^k}\frac1{( 1+|y-X^h|)^{\frac{n-2s}2+\tau}}\right).
\end{eqnarray}

For the second term on the right-hand side of \eqref{33}, we have
\begin{equation}\label{36}
\int_{\mathbb{R}^n}\frac1{|y-z|^{n-2s}}|g(z)|dz\leq \|g\|_{**}\int_{\mathbb{R}^n}\frac1{|y-z|^{n-2s}}\gamma(z)\sum_{h=1}^{(m+1)^k}\frac1{(1+|z-X^h|)^{\frac{n+2s}2 +\tau}}dz.
\end{equation}
Using Lemma \ref{lm10}, we obtain
\begin{equation}\label{35}
\int_{\mathbb{R}^n}\frac1{|y-z|^{n-2s}}\sum_{h=1}^{(m+1)^k}\frac1{(1+|z-X^h|)^{\frac{n+2s}2 +\tau}}dz\leq C\sum_{h=1}^{(m+1)^k}\frac1{(1+|z-X^h|)^{\frac{n-2s}2+\tau}}.
\end{equation}

Define $B_i:=B_{\lambda l}(X^i)$, $B_{i,m}:=B_{r_0}(X^i)$ with $r_0=\max\{\frac m4,1\}$ and
\[\Omega_i:=\{z\in\mathbb{R}^n:\;\;|z-X^i|=\min_{j=1,\dots,(m+1)^k}|z-X^j|\}.\]

Without loss of generality, we assume $y\in\Omega_1$. Make use of Lemma \ref{lm9} under different cases, we have
\begin{equation}\label{37}
\frac1{\lambda^{\tau-s}}\sum_{h=1}^{(m+1)^k}\frac1{(1+|y-X^h|)^{\frac n2}}\leq C\left(\frac{1+ |y-X^1|}{\lambda}\right)^{\tau-s}\sum_{h=1}^{(m+1)^k}\frac1{(1+|y-X^h|)^{\frac{n-2s}2+\tau}}.
\end{equation}
Using Lemma \ref{lm10} and \eqref{37}, we have
\begin{eqnarray}\label{34}
\nonumber&&\int_{\mathbb{R}^n}\frac1{|y-z|^{n-2s}}\min_{i=1,\dots,(m+1)^k}\left\{\left(\frac{1+|z-X^i|}{\lambda} \right)^{\tau-s}\right\} \sum_{h=1}^{( m+1)^k}\frac1{(1+|z-X^h|)^{\frac{n+2s}2+\tau}}dz \\ \nonumber
&\leq& \frac C{\lambda^{\tau-s}}\int_{\mathbb{R}^n}\frac1{|y-z|^{n-2s}}\sum_{h=1}^{(m+1)^k} \frac1{(1+|z-X^h|)^{\frac n2+2s}}dz \\\nonumber
&\leq& \frac C{\lambda^{\tau-s}}\sum_{h=1}^{(m+1)^k}\frac1{(1+|y-X^h|)^{\frac n2}} \\
&\leq& C\left(\frac{1+ |y-X^1|}{\lambda}\right)^{\tau-s}\sum_{h=1}^{(m+1)^k}\frac1{(1+|y-X^h|)^{\frac{n-2s}2+\tau}}.
\end{eqnarray}
From the definition of $\gamma(y)$, \eqref{36}, \eqref{35} and \eqref{34}, we know
\begin{equation}\label{66}
\int_{\mathbb{R}^n}\frac1{|y-z|^{n-2s}}|g(z)|dz\leq C\|g\|_{**}\gamma(y)\sum_{h=1}^{(m+1)^k}\frac1{(1+|y-X^h|)^{\frac{n-2s}2+\tau}}.
\end{equation}
Now \eqref{64} follows from \eqref{33}, \eqref{65} and \eqref{66}.

\end{proof}

Consider the following problem
\begin{equation}\label{17}
\left\{\begin{array}{ll}(-\Delta)^s\phi-\frac{n+2s}{n-2s}K(\frac x\lambda)W_m^{\frac{4s}{n-2s}}\phi=g+\sum_{i=1}^{(m+1)^k} \sum_{j=1}^{n+1}c^{(m)}_{ij}U_{P^i,\Lambda_i}^{\frac{4s}{n-2s}}Z_{i,j}, \\
\int_{\mathbb{R}^n}U_{P^i,\Lambda_i}^{\frac{4s}{n-2s}}Z_{i,j}\phi dx=0,\;\;\;\; \phi\in \dot{H}^s(\mathbb{R}^n), \;\;\;\;i=1,\dots,(m+1)^k,\;\;j=1,\dots,n+1,\end{array}\right.
\end{equation}
where $Z_{i,j}=\frac{\partial U_{P^i,\Lambda_i}}{\partial P^i_j}$ for $j=1,\dots,n$ and $Z_{i,n+1}= \frac{\partial U_{P^i,\Lambda_i}}{\partial \Lambda_i}$.
\begin{lemma}\label{lm15}
Assume $\phi$ solves the problem \eqref{17}, there is exists $l_0>0$, such that for all $l>l_0$, we have $\|\phi\|_*\leq C\|g\|_{**}$, where $C$ is independent of $m$.
\end{lemma}

\begin{proof}
If this lemma is not right, then there would be sequences $\{g_l\}_{l=1}^\infty$ and $\{\phi_l\}_{l=1}^\infty$ satisfying \eqref{17} with $\|\phi_l\|_*=1$ and $\|g_l\|_{**}\to 0$ as $l\to +\infty$. For notation simplicity, we suppress $l$ in the argument below.

First, we give an estimate of the parameters $c^{(m)}_{ij}$. Multiplying \eqref{17} with $Z_{r,t}$ and integrating on both sides, we get
\begin{equation}\label{18}
-\frac{n+2s}{n-2s}\int_{\mathbb{R}^n}K(\frac x\lambda)W_m^{\frac{4s}{n-2s}}\phi Z_{r,t}dx=\int_{\mathbb{R}^n}gZ_{r,t}dx+\sum_{i=1}^{(m+1)^k} \sum_{j=1}^{n+1}c^{(m)}_{ij}\int_{\mathbb{R}^n}U_{P^i,\Lambda_i}^{\frac{4s}{n-2s}}Z_{i,j}Z_{r,t}dx.
\end{equation}
For the first term on the right hand side of \eqref{18}, using Lemma \ref{lm10}, we have
\begin{eqnarray*}
\left|\int_{\mathbb{R}^n}g Z_{r,t}dx\right|&\leq& C\|g\|_{**}\int_{\mathbb{R}^n}\frac1{(1+|x-X^r|)^{n-2s}}\gamma(x)\sum_{j=1}^{ (m+1)^k}\frac1{(1+|x-X^j|)^{\frac{n+2s}2+\tau}}dx \\
&\leq& \frac C{\lambda^{\tau-s}}\|g\|_{**}\int_{\mathbb{R}^n}\frac1{(1+|x-X^r|)^{n-2s}}\sum_{j=1}^{ (m+1)^k}\frac1{(1+|x-X^j|)^{\frac n2+2s}}dx \\
&\leq& C\frac{\|g\|_{**}}{\lambda^{\tau-s}}\left(\int_{\mathbb{R}^n}\frac1{(1+|x-X^r|)^{n+\frac{n}2}}dx+ \sum_{j\not=r}\frac1{|X^j-X^r|^{ \frac{n}2}}\right) \\
&\leq& C\frac{\|g\|_{**}}{\lambda^{\tau-s}},
\end{eqnarray*}
where we have used the fact that
\begin{equation}\label{67}
\sum_{j\not=r}\frac1{|X^j-X^r|^{\frac n2}}\text{ converges for }\frac n2>k.
\end{equation}
Since the left hand side of the equation \eqref{18} is estimated in Lemma \ref{lm5}, we have
\begin{equation*}
\sum_{i=1}^{(m+1)^k} \sum_{j=1}^{n+1}c^{(m)}_{ij}\int_{\mathbb{R}^n}U_{P^i,\Lambda_i}^{\frac{4s}{n-2s}}Z_{i,j}Z_{r,t}dx= \frac1{\lambda^{\tau-s}}O\left( \|g\|_{**}+\frac{\|\phi\|_*}{(\lambda l)^{\frac n2}}\right).
\end{equation*}
As we know $\int_{\mathbb{R}^n}U_{P^i,\Lambda_i}^{\frac{4s}{n-2s}}Z_{i,j}Z_{i,t}dx=C\delta_{jt}$ and $\int_{\mathbb{R}^n}|U_{P^i,\Lambda_i}^{\frac{4s}{n-2s}}Z_{i,j}Z_{r,t}|dx\leq \frac{C}{|X^i-X^r|^{n-2s}}$ for $i\not=r$, we obtain
\[\max_{i,j}\{|c^{(m)}_{ij}|\}=\frac1{\lambda^{\tau-s}}O\left(\|g\|_{**}+\frac{\|\phi\|_*}{(\lambda l)^{\frac n2}}\right).\]

An argument similar to the one used in \eqref{37} yields
\begin{eqnarray*}
\left|\sum_{i=1}^{(m+1)^k} \sum_{j=1}^{n+1}c^{(m)}_{ij}U_{P^i,\Lambda_i}^{\frac{4s}{n-2s}}Z_{i,j}\right|&\leq& \frac C{\lambda^{\tau-s}}\left(\|g\|_{**}+\frac{\|\phi\|_*}{(\lambda l)^{\frac n2}}\right) \sum_{i=1}^{(m+1)^k} \frac1{(1+|y-X^i|)^{n+2s}} \\
&\leq& C\left(\|g\|_{**}+\frac{\|\phi\|_*}{(\lambda l)^{\frac n2}}\right)\gamma(y)\sum_{i=1}^{(m+1)^k} \frac1{(1+|y-X^i|)^{\frac{n+2s}2+\tau}}.
\end{eqnarray*}
From the definition of the norm $\|\cdot\|_{**}$, we have
\[\|\sum_{i=1}^{(m+1)^k} \sum_{j=1}^{n+1}c^{(m)}_{ij}U_{P^i,\Lambda_i}^{\frac{4s}{n-2s}}Z_{i,j}\|_{**}\leq C\left(\|g\|_{**}+\frac{\|\phi\|_*}{(\lambda l)^{\frac n2}}\right).\]

Applying Lemma \ref{lm6} to the first equation of the system \eqref{17}, one get
\begin{eqnarray*}
&&\left(\gamma(y)\sum_h\frac1{(1+|y-X^h|)^{\frac{n-2s}2+\tau}}\right)^{-1}|\phi(y)| \\
&\leq& C\left(\|g\|_{**}+\|\sum_{i=1}^{(m+1)^k} \sum_{j=1}^{n+1}c^{(m)}_{ij}U_{P^i,\Lambda_i}^{\frac{4s}{n-2s}}Z_{i,j}\|_{**}+\left(\frac1{(\lambda l)^{\frac{4s}{n-2s}k}}+\frac{\sum_h\frac1{(1+|y-X^h|)^{\frac{n-2s}2+\tau+\theta}}}{\sum_h\frac1{(1+|y-X^h|)^{ \frac{n-2s}2+\tau}}}\right)\|\phi\|_*\right) \\
&\leq& C\left(\|g\|_{**}+\frac1{(\lambda l)^{\frac{4s}{n-2s}k}}+\frac{\sum_h\frac1{(1+|y-X^h|)^{\frac{n-2s}2+\tau+\theta}}}{\sum_h\frac1{(1+|y-X^h|)^{ \frac{n-2s}2+\tau}}}\right).
\end{eqnarray*}
As a result, there exist a number $i_0\in \mathbb{N}$ and a large constant $R>0$, such that
\begin{equation}\label{21}
1=\|\phi\|_*=\sup_{B_R(X^{i_0})}(\gamma(y)\sum_{h=1}^{(m+1)^k}\frac1{(1+|y-X^h|)^{\frac{n-2s}2+\tau}} )^{-1}|\phi(y)|.
\end{equation}
Hence there is a constant $c_0>0$ such that $|\lambda^{\tau-s}\phi|_{L^{\infty}(B_R(X^{i_0}))}\geq c_0$.

Applying Lemma \ref{lm12} to the equation \eqref{17}, we know $\lambda^{\tau-s}\phi$ is equi-continuous. Also $\lambda^{\tau-s}|\phi(\cdot)|$ is uniformly bounded. In fact, we assume that $y\in \Omega_1$ with no loss of generality. From the fact \eqref{67}, we have
\[
\lambda^{\tau-s}|\phi(y)|\leq \|\phi\|_*\sum_h\frac1{(1+|y-X^h|)^{\frac n2}}\leq C+\sum_{h\not=1}\frac1{|X^h-X^1|^{\frac n2}}\leq C.
\]
Then the Arzel\`{a}-Ascoli Theorem yields that there is a function $\tilde\phi$, such that $\lambda^{\tau-s}\phi(\cdot+P^{i_0})$ convergent to $\tilde \phi$ uniformly on compact sets. Then
\begin{equation}\label{68}
|\tilde\phi|_{L^{\infty}(B_{R+1}(0))}\geq c_0.
\end{equation}
Using a similar argument as in \cite[Lemma 7.3]{Davila2015}, we know $\tilde\phi$ satisfies
\begin{equation*}\
\left\{\begin{array}{ll}
(-\Delta)^s\tilde \phi-\frac{n+2s}{n-2s}U_{0,\Lambda_{i_0}}^{\frac{4s}{n-2s}}\tilde \phi=0,\\
 \int_{\mathbb{R}^n} U_{0,\Lambda_{i_0}}^{\frac{4s}{n-2s}}\frac{\partial U_{0,\Lambda_{i_0}}}{\partial \Lambda_{i_0}}\tilde \phi=0,\hspace{3,mm} \int_{\mathbb{R}^n} U_{0,\Lambda_{i_0}}^{\frac{4s}{n-2s}}\frac{\partial U_{0,\Lambda_{i_0}}}{\partial P^{i_0}_j}\tilde \phi=0,\;\;j=1,\dots,n .\end{array}\right.
\end{equation*}
Then $\tilde\phi=0$ by nondegeneracy, which is contradict to \eqref{68}. Hence the solution $\phi$ of the equation \eqref{17} satisfies $\|\phi\|_*\leq C \|g\|_{**}$.

\end{proof}

Combining Lemma \ref{lm15}, Lemma \ref{lm12} and the argument of \cite[Proposition 4.1]{del_Pino2003}( cf. \cite[Proposition 2.2]{LiuCG2016}), we have
\begin{proposition}\label{prop4}
For any $g$ satisfying $\|g\|_{**}<+\infty$, \eqref{17} has a unique solution $\phi=L_m(g)\in \dot{H}^s(\mathbb{R}^n)\cap C^{0,\alpha}(\mathbb{R}^n)$ with $\alpha=\min\{2s,1\}$, such that $\|L_m(g)\|_*\leq C\|g\|_{**}$. The constant $c^{(m)}_{ij}$ satisfies $|c^{(m)}_{ij}|\leq \frac C{\lambda^{\tau-s}}\|g\|_{**}$.
\end{proposition}

Since we are interested in the solution of the form $W_m+\phi_m$ of the equation \eqref{2},  we now consider the following problem
\begin{equation}\label{22}
\left\{\begin{array}{ll}(-\Delta)^s\phi-\frac{n+2s}{n-2s}K(\frac x\lambda)W_m^{\frac{4s}{n-2s}}\phi=N(\phi)+l_m+\sum_{i=1}^{(m+1)^k} \sum_{j=1}^{n+1}c^{(m)}_{ij}U_{P^i,\Lambda_i}^{\frac{4s}{n-2s}}Z_{i,j}, \\
\int_{\mathbb{R}^n}U_{P^i,\Lambda_i}^{\frac{4s}{n-2s}}Z_{i,j}\phi dx=0,\;\;\;\; \phi\in \dot{H}^s(\mathbb{R}^n),\;\;\;\;i=1,\dots,(m+1)^k,\;\;j=1,\dots,n+1,\end{array}\right.
\end{equation}
where
\begin{equation*}
N(\phi)=K(\frac x\lambda)\left((W_m+\phi)^{\frac{n+2s}{n-2s}}_+-W_m^{\frac{n+2s}{n-2s}}-\frac{n+2s}{n-2s} W_m^{\frac{4s}{n-2s}}\phi\right)
\end{equation*}
and
\begin{equation*}
l_m=K(\frac x\lambda)W_m^{\frac{n+2s}{n-2s}}-\sum_{i=1}^{(m+1)^k}U_{P^i,\Lambda_i}^{ \frac{n+2s}{n-2s}}.
\end{equation*}
\begin{lemma}\label{lm7}
For the terms $N(\phi)$ and $l_m$ defined above, we have the following estimates
\[\|N(\phi)\|_{**}\leq C\|\phi\|_*^{\min\{2,2^*(s)-1\}},\]
\[\|l_m\|_{**}\leq \frac C{\lambda^{\frac{n+2s}2-\tau}}.\]
\end{lemma}
\begin{proof}
The proof of the first estimation is rather standard(cf. \cite[Lemma 2.4]{WeiYan2010} for ideas). We only prove the second estimate.

Without loss of generality, we assume $x\in \Omega_1$. Then
\begin{eqnarray}\label{38}
\nonumber l_m &=& K(\frac x\lambda)W_m^{\frac{n+2s}{n-2s}}-U_{P^1,\Lambda_1}^{\frac{n+2s}{n-2s}}-\sum_{h\not=1}U_{ P^h,\Lambda_h}^{ \frac{n+2s}{n-2s}} \\
&=& \left(K(\frac x\lambda)-1\right)U_{P^1,\Lambda_1}^{\frac{n+2s}{n-2s}}+O\left(\left(\sum_{h\not=1}U_{P^h,\Lambda_h}\right)^{ \frac{n+2s}{n-2s}}+U_{P^1,\Lambda_1}^{\frac{4s}{n-2s}}\sum_{h\not=1}U_{P^h,\Lambda_h}\right).
\end{eqnarray}

The two error terms in \eqref{38} can be estimated by using Lemma \ref{lm9} under different cases.

\noindent\textbf{Case 1:} $x\in \Omega_1\cap B_1^c\cap B_{1,m}$, we have $\gamma(x)=1$. Using Lemma \ref{lm9}, we have
\begin{eqnarray*}
\left(\sum_{h\not=1}U_{P^h,\Lambda_h}\right)^{\frac{n+2s}{n-2s}}&\leq& \frac C{(\lambda l)^{\frac{n+2s}{n-2s}k}} \frac1{(1+|x-X^1|)^{n+2s-\frac{n+2s}{n-2s}k}}  \\
&\leq& \frac C{(\lambda l)^{\frac{n+2s}2-\tau+k}}\frac1{(1+|x-X^1|)^{\frac{n+2s}2+\tau-k}} \\
&\leq& \frac C{(\lambda l)^{\frac{n+2s}2-\tau}}\sum_h\frac1{(1+|x-X^h|)^{\frac{n+2s}2+\tau}}.
\end{eqnarray*}

\noindent\textbf{Case 2:} $x\in \Omega_1\cap B_1$, it holds that $|x-X^i|\geq \frac12|X^i-X^1|\geq \frac12\lambda l$ for $i\not=1$. From Lemma \ref{lm9},
\begin{equation*}
\left(\sum_{h\not=1}U_{P^h,\Lambda_h}\right)^{\frac{n+2s}{n-2s}}\leq \left\{\begin{array}{ll}
\frac C{(\lambda l)^{\frac{n+2s}2-\tau}}\displaystyle\sum_h\frac1{(1+|x-X^h|)^{\frac{n+2s}2+\tau}}, &\mbox{if $x\in \Omega_1\cap B_1\cap B_{\lambda}^c(X_1)$,} \\
C\left(\frac{1+|x-X^1|}{\lambda}\right)^{\tau-s}\frac{\lambda^{\tau-s}}{(\lambda l)^{\frac n2}}\displaystyle\sum_h\frac1{(1+|x-X^h|)^{\frac{n+2s}2+\tau}}, &\mbox{if $x\in \Omega_1\cap B_1\cap B_{\lambda}(X^1)$.}
\end{array}\right.
\end{equation*}

\noindent\textbf{Case 3:} $x\in \Omega_1\cap B_{1,m}^c$, we can get
\begin{eqnarray*}
\left(\sum_{h\not=1}U_{P^h,\Lambda_h}\right)^{\frac{n+2s}{n-2s}}\leq \frac{Cm^{\frac{n+2s}{n-2s}k}}{(1+|x-X^1|)^{n+2s}}\leq \frac{Cm^{\frac{4s}{n-2s}k}}{(1+|x-X^1|)^{2s}}\frac{(1+C[\frac m2])^k}{(1+|x-X^1|)^n}.
\end{eqnarray*}
Following the proof of Lemma \ref{lm9}, we have
\begin{eqnarray*}
\sum_h\frac1{(1+|x-X^h|)^n}&\geq& \frac C{(1+|x-X^1|)^n}\left(1+2^{-k}\int_{[0,[\frac m2]+1]^k\char92 [0,1]^k}\frac1{(1+\frac{\lambda l}{1+|x-X^1|}|z|)^n}dz\right) \\
&\geq& \frac{(1+C[\frac m2])^k}{(1+|x-X^1|)^n}.
\end{eqnarray*}
Since in the domain $\Omega_1\cap B_{1,m}^c$, we can get $|x-X^h|\geq \frac12 \lambda l$ for $h=1,\dots,(m+1)^k$, then
\begin{eqnarray*}
\left(\sum_{h\not=1}U_{P^h,\Lambda_h}\right)^{\frac{n+2s}{n-2s}}\leq \frac C{(\lambda l)^{2s}}\sum_h\frac1{(1+|x-X^h|)^n}\leq \frac C{(\lambda l)^{\frac{n+2s}2-\tau}}\sum_h\frac1{(1+|x-X^h| )^{\frac{n+2s}2+\tau}}.
\end{eqnarray*}

Combining these three cases above, we have  $\|(\sum_{h\not=1}U_{P^h,\Lambda_h})^{\frac{n+2s}{n-2s}}\|_{**}\leq \frac C{(\lambda l)^{\frac{n+2s}2-\tau}}$. By the same procedure, we can also get the estimation $\|U_{P^1,\Lambda_1}^{\frac{4s}{n-2s}}\sum_{h\not=1} U_{P^h,\Lambda_h} \|_{**}\leq \frac C{(\lambda l)^{\frac{n+2s}2-\tau}}.$

At last, we estimate the first term in \eqref{38}.

In the case of $|x-X^1|\geq \lambda$, we have $\gamma(x)=1$. Then
\begin{equation}\label{39}
\left|K(\frac x\lambda)-1\right|U_{P^1,\Lambda_1}^{\frac{n+2s}{n-2s}}\leq \frac C{(1+|x-X^1|)^{n+2s}}\leq \frac C{\lambda^{\frac{n+2s}2-\tau}}\gamma(x)\sum_{h=1}^{(m+1)^k}\frac1{(1+|x-X^h|)^{\frac{n+2s}2+\tau}}.
\end{equation}
In the case of $|x-X^1|<\lambda$, it holds $\frac{1+|x-X^1|}{\lambda}\leq C$. The condition $(H_3)$ yields
\begin{eqnarray}\label{40}
\nonumber&&\left|K(\frac x\lambda)-1\right|U_{P^1,\Lambda_1}^{\frac{n+2s}{n-2s}}\leq C\frac{|x-X^1|^\beta}{\lambda^\beta}\frac1{(1+|x-X^1|)^{n+2s}} \\ \nonumber
&\leq&\frac C{\lambda^{\frac{n+2s}2-\tau}} \left(\frac{1+|x-X^1|}{\lambda}\right)^{\tau-s}\sum_{h=1}^{(m+1)^k}\frac1{(1+|x-X^h|)^{\frac{n+2s}2+\tau}} \\
&\leq& \frac C{\lambda^{\frac{n+2s}2-\tau}} \gamma(y)\sum_{h=1}^{(m+1)^k}\frac1{(1+|x-X^h|)^{\frac{n+2s}2+\tau}}.
\end{eqnarray}
Summarizing \eqref{39} and \eqref{40}, we have
\begin{equation*}
\|\left(K(\frac x\lambda)-1\right)U_{P^1,\Lambda_1}^{\frac{n+2s}{n-2s}}\|_{**}\leq \frac{C}{\lambda^{\frac{n+2s}2-\tau}}.
\end{equation*}
Hence this lemma follows.

\end{proof}

\begin{proposition}\label{prop1}
For $\lambda$ large enough, the problem \eqref{22} has a unique solution $\phi_m\in \dot{H}^s(\mathbb{R}^n)\cap C^{0,\alpha}(\mathbb{R}^n)$ with $\alpha=\min\{2s,1\}$, such that $\|\phi_m\|_*\leq \frac C{\lambda^{\frac{n+2s}2-\tau}}$. The constants $c_{ij}^{(m)}$ satisfy $|c^{(m)}_{ij}|\leq  C\lambda^{-\frac n2}$.
\end{proposition}
\begin{proof}
We define
\[E=\left\{\varphi\in \dot{H}(\mathbb{R}^n)\cap C(\mathbb{R}^n):\;\|\varphi\|_*\leq \frac1{\lambda^{\frac{n+2s}2-\tau-\epsilon_1}},\;\;\int_{\mathbb{R}^n}U_{P^i,\Lambda_i}^{\frac{4s }{n-2s}}Z_{i,j}\varphi=0,\;\;\begin{array}{ll}i=1,\dots,(m+1)^k,\\j=1,\dots,n+1\end{array}\right\},\]
where $\epsilon_1=\min\{\frac 14,\frac{2s}{n+2s}\}(\frac{n+2s}2-\tau)$. Notice that $(E,\|\cdot\|_*)$ is a metric space.

In order to use the contraction map theorem, we define $A\varphi:=L_m(N(\varphi)+l_m)$, where $L_m$ is an operator defined in Proposition \ref{prop4}.

Firstly, we show that $A$ maps $E$ into itself for $\lambda$ large. Combining Proposition \ref{prop4} and Lemma \ref{lm7}, we have $\forall\varphi\in E$,
\[\|A\varphi\|_*\leq C(\|N(\varphi)\|_{**}+\|l_m\|_{**})\leq C(\|\varphi\|_*^{\min\{2,2^*(s)-1\}}+\|l_m\|_{**})\leq \frac1{\lambda^{\frac{n+2s}2-\tau-\epsilon_1}}.\]

Secondly, we  prove $A$ is an contraction map for $\lambda$ large.

Choose $\varphi_1,\varphi_2\in E$ with $\varphi_1\not=\varphi_2$. If $N\geq 6s$, we have
\begin{eqnarray*}
&&|N(\varphi_1)-N(\varphi_2)|=|N^\prime(t\varphi_1+(1-t)\varphi_2)(\varphi_1-\varphi_2)| \\
&\leq& C(|\varphi_1^{\frac{4s}{n-2s}}|+|\varphi_2^{\frac{4s}{n-2s}}|)(|\varphi_1-\varphi_2|) \\
&\leq& C(\|\varphi_1\|_*^{\frac{4s}{n-2s}}+\|\varphi_2\|_*^{\frac{4s}{n-2s}})\|\varphi_1-\varphi_2\|_*\left( \gamma(x)\sum_{h=1}^{(m+1)^k}\frac1{(1+|x-X^h|)^{\frac{n-2s}2+\tau}}\right)^{\frac{n+2s}{n-2s}} \\
&\leq& C(\|\varphi_1\|_*^{\frac{4s}{n-2s}}+\|\varphi_2\|_*^{\frac{4s}{n-2s}})\|\varphi_1-\varphi_2\|_*\left(\gamma(x) \sum_{ h=1}^{(m+1)^k}\frac1{(1+|x-X^h|)^{\frac{n+2s}2+\tau}}\right).
\end{eqnarray*}
We remind that in the last inequality, we have used the H\"{o}lder inequality.
Hence $\|N(\varphi_1)-N(\varphi_2)\|_{**}\leq  C(\|\varphi_1\|_*^{\frac{4s}{n-2s}}+\|\varphi_2\|_*^{\frac{4s}{n-2s}}) \|\varphi_1-\varphi_2\|_*$.

In the case of $N\leq 6s$, we also have $\|N(\varphi_1)-N(\varphi_2)\|_{**}\leq  C(\|\varphi_1\|_*^{\min\{\frac{4s}{n-2s},1\}}+\|\varphi_2\|_*^{\min\{\frac{4s}{n-2s},1\}}) \|\varphi_1-\varphi_2\|_*$ by a similar argument.

Then there hold
\begin{eqnarray*}
\|A\varphi_1-A\varphi_2\|_*&\leq& C\|N(\varphi_1)-N(\varphi_2)\|_{**} \\
&\leq&  C(\|\varphi_1\|_*^{\min\{\frac{4s}{n-2s},1\}}+\|\varphi_2\|_*^{\min\{\frac{4s}{n-2s},1\}}) \|\varphi_1-\varphi_2\|_* \\
&\leq& \frac C{\lambda^{(\frac{n+2s}2-\tau-\epsilon_1)\min\{\frac{4s}{n-2s},1\}}}\|\varphi_1-\varphi_2\|_*.
\end{eqnarray*}
For $\lambda$ large enough, we get $\|A\varphi_1-A\varphi_2\|_*\leq \frac12\|\varphi_1-\varphi_2\|_*$.

Using the contracting map theorem, we know that there is a unique $\phi_m\in E$, such that $A(\phi_m)=\phi_m$, \textit{i.e.} $\phi_m$ is a unique solution of \eqref{22} in $E$. From Proposition \ref{prop4} and Lemma \ref{lm7}, we know $\phi_m\in \dot{H}^s(\mathbb{R}^n)\cap C^{0,\alpha}(\mathbb{R}^n)$ satisfying $\|\phi_m\|_*=\|A\phi_m\|_*\leq\frac C{\lambda^{\frac{n+2s}2-\tau}}$ and $|c_{ij}^{(m)}|\leq C \lambda^{-\frac n2}$, since $(\frac{n+2s}2-\tau-\epsilon_1)\min\{2,2^*(s)-1\}>\frac{n+2s}2-\tau$.

\end{proof}

\section{Proof of the main theorem}\label{sect3}
Let $\Lambda:=(\Lambda_1,\dots,\Lambda_{(m+1)^k})\in \mathbb{R}_+^{(m+1)^k}$ and $P:=(P^1,\dots,P^{(m+1)^k})\in \mathbb{R}^{n\times (m+1)^k}$, in which $P^i=(P^i_1,\dots,P^i_n)\in \mathbb{R}^n$ for $i=1,2\dots,(m+1)^k$. We define $J(P,\Lambda)=I(W_m+\phi_m)$, where $\phi_m$ is a unique small solution obtained by Proposition \ref{prop1}. A standard argument shows that from a critical point of $J$, we can get  a critical point of $I$ of the form $W_m+\phi_m$(for example, cf. \cite[Lemma 6.1]{del_Pino2003} for ideas).

\begin{proposition}\label{prop5}
For $\lambda$ large, we have the following expansions
\begin{equation}\label{23}
\frac{\partial J}{\partial P^i_j}(P,\Lambda)=-\frac{c_3a_j}{\Lambda_i^{\beta-2}\lambda^\beta}(P^i_j-X^i_j) +O\left(\frac{|P^i-X^i|^2}{\lambda^\beta}\right)+o(\lambda^{-\beta}),
\end{equation}
and
\begin{equation}\label{24}
\frac{\partial J}{\partial \Lambda_i}(P,\Lambda)=-\frac{c_1}{\Lambda_i^{\beta+1}\lambda^\beta}+\sum_{h\not=i}\frac{c_2}{ \Lambda_i(\Lambda_i\Lambda_h)^{\frac{n-2s}2}|X^i-X^h|^{n-2s}}+O\left(\frac{|P^i-X^i|^{\min\{2,\beta-1\}}}{ \lambda^\beta}\right)+o(\lambda^{-\beta}),
\end{equation}
where $i=1,\dots,(m+1)^k$ and $j=1,\dots,n$ and the constant $c_1,c_2,c_3$ are positive.
\end{proposition}

\begin{proof}
A simple calculation yields
\begin{eqnarray*}
\nonumber\frac{\partial J}{\partial P^i_j}(P,\Lambda)&=&\langle I^\prime(W_m+\phi_m),\frac{\partial U_{P^i,\Lambda_i}}{\partial P^i_j}+\frac{\partial \phi_m}{\partial P^i_j}\rangle \\
\nonumber&=& \frac{\partial I}{\partial P^i_j}(W_m)+\int_{\mathbb{R}^n}K(\frac x\lambda)(W_m^{\frac{n+2s}{n-2s}}-(W_m+\phi_m)_+^{\frac{n+2s}{n-2s}})\frac{\partial U_{P^i,\Lambda_i}}{\partial P^i_j} \\
&&+\sum_{t=1}^{(m+1)^k}\sum_{h=1}^{n+1}c^{(m)}_{th}\int_{\mathbb{R}^n}U_{P^t,\Lambda_t}^{\frac{4s}{n-2s}}Z_{t,h} \frac{\partial \phi_m}{\partial P^i_j}.
\end{eqnarray*}
The functional $\frac{\partial I}{\partial P^i_j}(W_m)$ is expanded in the Proposition \ref{prop2}. So we only need to estimate the last two terms in the eqality above.

From Lemma \ref{lm13}, we see that for $\lambda$ large enough, $\{x: W_m\leq -\phi_m\}\subset\{x: \frac12W_m\leq |\phi_m|\}\subset \cup_h(\Omega_h\cap B_h^c)$. Then we have
\begin{eqnarray*}
&&\int_{\mathbb{R}^n}K(\frac x\lambda)((W_m+\phi_m)_+^{\frac{n+2s}{n-2s}}-W_m^{\frac{n+2s}{n-2s}})\frac{\partial U_{P^i,\Lambda_i}}{\partial P^i_j}dx\\
&=&\int_{\mathbb{R}^n}K(\frac x\lambda)((W_m+\phi_m)^{\frac{n+2s}{n-2s}} -W_m^{\frac{n+2s}{n-2s}})\frac{\partial U_{P^i,\Lambda_i}}{\partial P^i_j}dx-\int_{-\phi_m\geq W_m}K(\frac x\lambda)(W_m+\phi_m)^{\frac{n+2s}{n-2s}}\frac{\partial U_{P^i,\Lambda_i}}{\partial P^i_j}dx \\
&=& \frac{n+2s}{n-2s}\int_{\mathbb{R}^n}K(\frac x\lambda)W_m^{\frac{4s}{n-2s}}\phi_m\frac{\partial U_{P^i,\Lambda_i}}{\partial P^i_j}dx+O\left(\int_{|\phi_m|\geq\frac12W_m} |\phi_m|^{\frac{n+2s}{n-2s}}U_{P^i,\Lambda_i}dx \right. \\
&&\left.+\int_{|\phi_m|<\frac12W_m} W_m^{\frac{6s-n}{n-2s}}\phi_m^2\frac{\partial U_{P^i,\Lambda_i}}{\partial P^i_j}dx\right) \\
&=&\frac{n+2s}{n-2s}\int_{\mathbb{R}^n}K(\frac x\lambda)W_m^{\frac{4s}{n-2s}}\phi_m\frac{\partial U_{P^i,\Lambda_i}}{\partial P^i_j}dx+O\left(\int_{\cup_h (\Omega_h\cap B_h^c)}|\phi_m|^{\frac{n+2s}{n-2s}}U_{P^i,\Lambda_i}dx\right. \\
&&\left.+\int_{\mathbb{R}^n} W_m^{\frac{6s-n}{n-2s}}\phi_m^2U_{P^i,\Lambda_i}dx\right).
\end{eqnarray*}
Using Proposition \ref{prop1}, Lemma \ref{lm5} and Lemma \ref{lm14}, we have
\begin{equation*}
\int_{\mathbb{R}^n}K(\frac x\lambda)((W_m+\phi_m)_+^{\frac{n+2s}{n-2s}}-W_m^{\frac{n+2s}{n-2s}})\frac{\partial U_{P^i,\Lambda_i}}{\partial P^i_j}dx=O\left(\lambda^{-n}\right)=o\left(\lambda^{-\beta}\right).
\end{equation*}
By using the orthogonal condition of \eqref{22} and Lemma \ref{lm4}, we have
\begin{eqnarray*}
&&\left|\sum_{t=1}^{(m+1)^k}\sum_{h=1}^{n+1}c_{th}\int_{\mathbb{R}^n}U_{P^t,\Lambda_t}^{\frac{4s}{n-2s}}Z_{t,h} \frac{\partial \phi_m}{\partial P^i_j}dx\right|=\left|\sum_{h=1}^{n+1}c_{ih}\int_{\mathbb{R}^n}\frac{\partial }{\partial P^i_j}(U_{P^i,\Lambda_i}^{\frac{4s}{n-2s}}Z_{i,h})\phi_mdx\right| \\
&\leq& \frac C{\lambda^{\frac n2}}\frac{\|\phi_m\|_*}{\lambda^{\tau-s}}\int_{\mathbb{R}^n}\frac1{(1+|x-X^i|)^{ n+2s}}\sum_{r=1}^{(m+1)^k}\frac1{(1+|x-X^r|)^{\frac n2}}dx\leq C\lambda^{-n}.
\end{eqnarray*}
Hence we can get \eqref{23}. The estimation \eqref{24} can be derived by the same procedure along with Proposition \ref{prop3}.

\end{proof}

\begin{remark}
From Proposition \ref{prop5}, we know that there  exist bounded functions $\Xi^i_j=\Xi^i_j(P,\Lambda,\lambda)$ and $\Theta^i_j=\Theta^i_j(P,\Lambda,\lambda)$, $i=1,\dots,(m+1)^k$; $j=1,\dots,n+1$ satisfying
\[|\Xi^i_j|\leq C, \text{ where $C$ is constant independent of $m$}\]
and
\[|\Theta^i_j|\leq C_\lambda, \text{ where $C_\lambda$ is a constant only depend on $\lambda$, and $C_\lambda\to 0$ as $\lambda\to \infty$}\]
such that
\begin{equation*}
\frac{\partial J}{\partial P^i_j}(P,\Lambda)=-\frac{c_3a_j}{\Lambda_i^{\beta-2}\lambda^\beta}(P^i_j-X^i_j) +\frac{|P^i-X^i|^2}{\lambda^\beta}\Xi^i_j+\lambda^{-\beta}\Theta^i_j,
\end{equation*}
and
\begin{equation*}
\frac{\partial J}{\partial \Lambda_i}(P,\Lambda)=-\frac{c_1}{\Lambda_i^{\beta+1}\lambda^\beta}+\sum_{h\not=i}\frac{c_2}{ \Lambda_i(\Lambda_i\Lambda_h)^{\frac{n-2s}2}|X^i-X^h|^{n-2s}}-\frac{|P^i-X^i|^{\min\{2,\beta-1\}}}{ \lambda^\beta}\Xi^i_{n+1}-\lambda^{-\beta}\Theta^i_{n+1}.
\end{equation*}
\end{remark}

\noindent\textit{Proof of Theorem \ref{th2}.}
Firstly we look for the solution of \eqref{2} of the form $W_m+\phi_m$, $m<\infty$. It is equivalent to solving the system $\left\{\begin{array}{ll}\frac{\partial J}{\partial P^i_j}(P,\Lambda)=0, \\\\ \frac{\partial J}{\partial \Lambda_i}(P,\Lambda)=0,\end{array}\right.\;i=1,2,\dots,(m+1)^k;\;\;j=1,2,\dots,n$, that is
\begin{equation}\label{25}
\left\{\begin{array}{ll}
\frac{c_3a_j}{\Lambda_i^{\beta-2}\lambda^\beta}(P^i_j-X^i_j)=\frac{|P^i-X^i|^2}{\lambda^\beta}\Xi^i_j+ \lambda^{-\beta}\Theta^i_j; \\
-\frac{c_1}{\Lambda_i^{\beta+1}\lambda^\beta}+\sum_{h\not=i}\frac{c_2}{\Lambda_i(\Lambda_i\Lambda_h)^{ \frac{n-2s}2 }|X^i-X^h|^{n-2s}}=\frac{|P^i-X^i|^{\min\{2,\beta-1\}}}{ \lambda^\beta}\Xi^i_{n+1}+\lambda^{-\beta}\Theta^i_{n+1}.
\end{array}\right.
\end{equation}

To simplify the equations \eqref{25}, we denote $d_j=\Lambda_j^{-\frac{n-2s}2}$ and $A_{ih}=\left\{ \begin{array}{ll} 0,&\mbox{if $i=h$}, \\ \frac{(\lambda l)^{n-2s}}{|X^i-X^h|^{n-2s}}, &\mbox{if $i\not=h$}.\end{array}\right.$  The equations \eqref{25} can be written as
\begin{equation}\label{26}
\left\{
\begin{array}{ll} P^i_j-X^i_j=\frac{\Lambda_i^{\beta-2}\Xi^i_j}{c_3a_j}|P^i-X^i|^2+\frac{\Lambda_i^{\beta-2} }{c_3a_j }\Theta^i_j,\\
c_2\sum_{h\not=i}A_{ih}d_h-c_1d_i^{\frac{2\beta}{n-2s}-1}=\Lambda_i^{\frac{n-2s}2+1}\Xi^i_{n+1} |P^i-X^i|^{\min\{2,\beta-1\}}+\Lambda_i^{\frac{n-2s}2+1}\Theta^i_{n+1},\end{array} \right.
\end{equation}
where $i=1,2,\dots,(m+1)^k$ and $j=1,2,\dots,n$.

Define a function $F(z):=\frac{c_2}{2}\sum_{h\not=i}A_{ih}z_iz_h-\frac{(n-2s)c_1}{2\beta}\sum_h z_h^{\frac{2\beta}{n-2s}}$, where $z=(z_1,z_2,\dots,z_{(m+1)^k})\in \mathbb{R}^{(m+1)^k}$. Obviously, $F(z)$ has a maximum point $b=(b_1,b_2,\dots,b_{(m+1)^k})\in \mathbb{R}_+^{(m+1)^k}$. It holds that
\begin{equation}\label{27}
c_2\sum_{h\not=i}A_{ih} b_h-c_1b_i^{\frac{2\beta}{n-2s}-1}=0,\;\;\;\;i=1,\dots,(m+1)^k.
\end{equation}

\noindent\textbf{Claim: }Each component $b_i$ of $b$ satisfies $0<C_1^\prime\leq b_i\leq C_2^\prime$ for some constant $C_1^\prime$ and $C_2^\prime$.

Suppose that $b_1\leq b_i\leq b_2$. Using the definition of $A_{ih}$, we know $\sum_{h\not=i}A_{ih}$ is bounded. From \eqref{27}, we can get
\[c_1b_2^{\frac{2\beta}{n-2s}-1}=c_2\sum_{h\not=2}A_{2h}b_h\leq C_3 b_2,\]
which tell us $b_2$ is bounded from above.

Using \eqref{27} again, we have
\[c_1b_1^{\frac{2\beta}{n-2s}-1}=c_2\sum_{h\not=1}A_{1h}b_h\geq c_2\sum_{h\not=1}A_{1h}b_1\geq c_2A_{12}b_1,\]
which implies $b_1$ is bounded from below, away from zero. Hence the Claim follows.

We can choose a small $\delta_0>0$ such that $b_2^{-\frac2{n-2s}}-\delta_0>0$. The constant $C_1$ and $C_2$ in the introduction can be defined by
\begin{equation}\label{63}
C_1=b_2^{-\frac2{n-2s}}-\delta_0\;\; \text{and} \;\;C_2=b_1^{-\frac2{n-2s}}+\delta_0.
\end{equation}

For any $x=(x_1,\dots,x_{(m+1)^k})\in \mathbb{R}^{(m+1)^k}$, we denote $\|x\|_0=\max_j\{|x_j|\}$.  Let $|\frac{x_{i_0}}{b_{i_0}}|=\|\frac{x}{b}\|_0$. From the claim above, we know $|x_{i_0}|\geq C\|x\|_0$. Using \eqref{27}, we have
\begin{eqnarray*}
|(D^2F(b)x)_{i_0}|&=&|c_2\sum_{h\not=i_0}A_{i_0h}x_h-c_1(\frac{2\beta}{n-2s}-1)b_{i_0}^{\frac{2\beta}{n-2s}-2} x_{i_0}| \\
&\geq& c_1(\frac{2\beta}{n-2s}-1)b_{i_0}^{\frac{2\beta}{n-2s}-2}|x_{i_0}|-c_2|\sum_{h\not=i}A_{i_0h}x_h| \\
&\geq& c_1(\frac{2\beta}{n-2s}-1)b_{i_0}^{\frac{2\beta}{n-2s}-1}|\frac{x_{i_0}}{b_{i_0}}| -c_2\sum_{h\not=i}A_{i_0h}b_h|\frac{x_{i_0}}{b_{i_0}}| \\
&=& c_1(\frac{2\beta}{n-2s}-2)b_{i_0}^{\frac{2\beta}{n-2s}-2}|x_{i_0}|\geq C_4\|x\|_0.
\end{eqnarray*}
From the definition of $\|\cdot\|_0$, we get $\|D^2F(b)x\|_0\geq C_4\|x\|_0$.

Let $\theta=(\theta_1,\dots,\theta_{(m+1)^k})\in \mathbb{R}^{(m+1)^k}$ whose component $\theta_i:=d_i-b_i$, $i=1,\dots,(m+1)^k$. We define $X:=(X^1,\dots,X^{(m+1)^k})\in \mathbb{R}^{n\times(m+1)^k}$, in which $X^i=(X^i_1,\dots,X^i_n)\in \mathbb{R}^n$ for $i=1,\dots,(m+1)^k$. For any $Y=(Y^1,\dots,Y^{(m+1)^k})\in \mathbb{R}^{n\times (m+1)^k}$, we use the notation $\|Y\|:=\max_{i=1,\dots,(m+1)^k}\{|Y^i|\}$ to denote the maximum norm.

To simplify the equations \eqref{26}, we need to define some vector value functions below.
Let $\Xi^{(1)}:=\Xi^{(1)}(P,\Lambda,\lambda)\in \mathbb{R}^{n\times (m+1)^k}$ and $\Theta^{(1)}:=\Theta^{(1)}(P,\Lambda,\lambda)\in \mathbb{R}^{n\times (m+1)^k}$ with their exponents defined by
\[(\Xi^{(1)})^i_j=\frac{\Lambda_i^{\beta-2}\Xi^i_j|P^i-X^i|^2}{c_3a_j\|P-X\|^2}\;\;\;\; \text{and}\;\;\;\;
(\Theta^{(1)})^i_j=\frac{\Lambda_i^{\beta-2}}{c_3a_j}\Theta^i_j,\;\;\;\;i=1,\dots,(m+1)^k;j=1,\dots,n.\]
Let $\Xi^{(2)}:=\Xi^{(2)}(P,\Lambda,\lambda)\in \mathbb{R}^{(m+1)^k}$, $\Theta^{(2)}:=\Theta^{(2)}(P,\Lambda,\lambda)\in \mathbb{R}^{(m+1)^k}$ with exponents defined by
\[(\Xi^{(2)})^i=\Lambda_i^{\frac{n-2s}2+1}\Xi^i_{n+1}\frac{|P^i-X^i|^{\min\{2,\beta-1\}}}{\|P-X\|^{\min\{2, \beta-1\}}}\;\;\;\; \text{and}\;\;\;\;(\Theta^{(2)})^i=\Lambda_i^{\frac{n-2s}2+1}\Theta_{n+1}^i,\;\;\;\;i=1,\dots,(m+1)^k.\]
Define $\Pi(\theta):=(\Pi(\theta)^1,\dots,\Pi(\theta)^{(m+1)^k} )$, where $\Pi(\theta)^i(\;i=1,\dots,(m+1)^k)$ is defined by
\[(\Pi(\theta))^i=\int_0^1(D^3F(b+s\theta)\theta,\theta)_i(1-s)ds=\int_0^1 c_1(\frac{2\beta}{n-2s}-1)(\frac{2\beta}{n-2s}-2)(b_i+s\theta_i)^{\frac{2\beta}{n-2s}-3}\theta_i^2.\]

From their definition, we know there is a constant $C$ and a constant $C_\lambda$ satisfying $C_\lambda\to 0$ as $\lambda\to\infty$ such that  $\|\Xi^{(1)}\|\leq C$, $\|\Xi^{(2)}\|_0\leq C$; $\|\Theta^{(1)}\|\leq C_\lambda$ and $\|\Theta^{(2)}\|_0\leq C_\lambda$.

Using these notations and Taylor expansion, we can write the equations \eqref{26} into another form:
\begin{equation}\label{28}
\left\{\begin{array}{ll}
P-X=\|P-X\|^2\Xi^{(1)}+\Theta^{(1)};\\
D^2F(b)\theta=\|P-X\|^{\min\{2,\beta-1\}}\Xi^{(2)}+\Theta^{(2)}+\Pi(\theta),\end{array}\right.
\end{equation}

Let
\[
B=\left(\prod_{i=1}^{(m+1)^k}B_{2C_\lambda}(X^i)\right)\times B_{3C_4^{-1}C_\lambda}(0)\in \mathbb{R}^{n\times(m+1)^k }\times \mathbb{R}^{(m+1)^k}.
\]
Define a function
\begin{eqnarray*}
G:\;\;\;\;B\;\;&\to& B \\
(P,\theta)&\mapsto&(X+\Xi^{(1)}\|P-X\|^2+\Theta^{(1)},D^2F(b)^{-1}(\|P-X\|^{\beta-1}\Xi^{(2)}+\Theta^{(2)} +\Pi(\theta)))
\end{eqnarray*}
For each $(P,\theta)\in B$, Choose $C_\lambda$ small enough, we have
\[\|\Xi^{(1)}\|P-X\|^2+\Theta^{(1)}\|\leq C(2C_\lambda)^2+C_\lambda\leq 2C_\lambda^2, \]
and
\[\|D^2F(b)^{-1}(\|P-X\|^{\min\{2,\beta-1\}}\Xi^{(2)}+\Pi(\theta))\|_0\leq C_4^{-1}(CC_\lambda^{\min\{2,\beta-1\}}+C_\lambda+C\theta^2)\leq 3C_4^{-1}C_\lambda.\]

Since $C_\lambda\to 0$ as $\lambda\to\infty$, so for $\lambda$ large enough, we use the Brouwer fixed-point theorem to get a solution $(P^1,\dots,P^{(m+1)^k},\theta)$ of \eqref{28} in $B$. It holds that
\[|P^i-X^i|\leq 2C_\lambda\;\;\;\;\text{and}\;\;\;\;|\theta_i|=|b_i-\Lambda_i^{-\frac{n-2s}2}|\leq 3C_4^{-1} C_\lambda.\]
Hence we find a critical point of $I$ of the form $u_m:=W_m+\phi_m$ with $m<\infty$.

Next, we prove $u_m$ is a positive function. Denote $u_m^-=\min\{0,u_m\}$ and $u_m^+=u_m-u_m^-$. Then we have
\begin{equation*}
\int_{\mathbb{R}^n}(-\Delta)^su_m(x)u_m^-(x)dx=\int_{\mathbb{R}^n} K(\frac x\lambda)u_m^{\frac{n+2s}{n-2s}}(x)u_m^-(x)dx.
\end{equation*}
From the definition of $(-\Delta)^s$,

\begin{eqnarray*}
\int_{\mathbb{R}^n}(-\Delta)^su_mu_m^-dx &=&\int_{\mathbb{R}^n}(-\Delta)^su_m^-u_m^-dx+\int_{\mathbb{R}^n} (-\Delta)^s u_m^+u_m^-dx \\
&=&\int_{\mathbb{R}^n}|(-\Delta)^{\frac s2}u_m^-|^2dx+\int_{\mathbb{R}^n}\int_{\mathbb{R}^n}\frac{(u_m^+(x)- u_m^+(y)) u_m^-(x)}{|x-y|^{n+2s}}dxdy \\
&\geq& \int_{\mathbb{R}^n}|(-\Delta)^{\frac s2}u_m^-|^2dx.
\end{eqnarray*}
The Hardy-Littlewood-Sobolev inequality yields
\begin{eqnarray*}
\left(\int_{\mathbb{R}^n}|u_m^-|^{\frac{2n}{n-2s}}dx\right)^{\frac{n-2s}{n}}&\leq& C \int_{\mathbb{R}^n}|(-\Delta)^{\frac s2}u_m^-|^2dx \\
&\leq& C\int_{\mathbb{R}^n} K(\frac x\lambda)u_m^{\frac{n+2s}{n-2s}}(x)u_m^-(x)dx\leq C\int_{\mathbb{R}^n}|u_m^-|^{\frac{2n}{n-2s}}dx.
\end{eqnarray*}
Suppose $u_m^-\not\equiv 0$, we have $\int_{\mathbb{R}^n}|u_m^-|^{\frac{2n}{n-2s}}dx\geq C$. It is easy to get $u_m^-\leq |\phi_m|$. From this fact,
\begin{eqnarray*}
\int_{\mathbb{R}^n}|u_m^-|^{\frac{2n}{n-2s}}dx&\leq& \int_{\mathbb{R}^n}|\phi_m|^{\frac{2n}{n-2s}} dx\\
&\leq& \|\phi_m\|^{\frac{2n}{n-2s}}_*\int_{\mathbb{R}^n} \left(\gamma(x)\sum_{h=1}^{(m+1)^k}\frac1{(1+|x-X^h|)^{ \frac{n-2s}2+\tau}}\right)^{\frac{2n}{n-2s}}dx \\
&\leq& \|\phi_m\|^{\frac{2n}{n-2s}}_*C(m,k)\sum_h\int_{\mathbb{R}^n}\frac1{(1+|x-X^h|)^{n+\frac{2n}{n-2s}\tau}}dx.
\end{eqnarray*}
So we get $C\leq \int_{\mathbb{R}^n}|u_m^-|^{\frac{2n}{n-2s}}dx\leq C(m,k)\|\phi_m\|_*\to 0$ as $\lambda\to\infty$, which is impossible. Hence $u_m\geq 0$. Suppose there is a point $x_0$ such that $u_m(x_0)=0$, then
\begin{equation*}
0=K(\frac{x_0}\lambda)u_m^{\frac{n+2s}{n-2s}}(x_0)=(-\Delta)^s u_m(x_0)=P.V.\int_{\mathbb{R}^n}\frac{u_m(x_0)-u_m(y)}{|x_0-y|^{n+2s}}dy= P.V.\int_{\mathbb{R}^n}\frac{-u_m(y)}{|x_0-y|^{n+2s}}dy.
\end{equation*}
Then $u_m\equiv 0$ which is impossible. Hence $u_m>0$.

According to Proposition \ref{prop1}, $u_m=W_m+\phi_m\in C^{0,\alpha}(\mathbb{R}^n)\cap \dot{H}^s(\mathbb{R}^n)$. Using local Schauder estimate \cite[Proposition 2.11]{JinTianling2011} and a bootstrap argument, we know $u_m\in C^{2,\alpha^\prime}_{loc}(\mathbb{R}^n)$, for some $\alpha^\prime\in(0,1)$.

What is more, $|u_m|_{L^\infty}(\mathbb{R}^n)\leq C$ with $C$ independent of $m$. In fact, Choosing $x\in \Omega_1$ with no loss of generality, we have
\begin{equation*}
W_m(x)\leq \sum_{i=1}^{(m+1)^k}\frac C{(1+|x-X^i|)^{n-2s}}\leq C+\sum_{i\not=1}^{(m+1)^k}\frac C{|X^i-X^1|^{n-2s}}\leq C+\frac C{(\lambda l)^{n-2s}}\leq C,
\end{equation*}
and
\begin{equation*}
|\phi_m|\leq \|\phi_m\|_*\sum_{h=1}^{(m+1)^k}\frac1{(1+|x-X^h|)^{\frac{n-2s}2+\tau}}\leq \|\phi_m\|_*\sum_{h=1}^\infty\frac1{(1+|x-X^h|)^{\frac{n-2s}2+\tau}}\leq C\|\phi_m\|_*\leq C.
\end{equation*}
Since $\phi_m$ satisfies the equation
$(-\Delta)^s\phi_m-\frac{n+2s}{n-2s}K(\frac x\lambda)W_m^{\frac{4s}{n-2s}}\phi_m=N(\phi_m)+l_m$, then from Lemma \ref{lm12}, Lemma \ref{lm7} and Proposition \ref{prop1}, we know that for any $x,y\in \mathbb{R}^n$ with $x\not=y$, there holds
\begin{equation*}
\frac{|\phi_m(x)-\phi_m(y)|}{|x-y|^\alpha}\leq \frac{C}{\lambda^{\tau-s}}\max\{\|\phi_m\|_*,\|N(\phi_m)\|_{**}+\|l_m\|_{**}\}\leq C,\;\;\text{where}\;\;\alpha=\min\{1,2s\}.
\end{equation*}
Also from simple calculation, we get for any $x\in \mathbb{R}^n$ and $R>0$, $\|W_m\|_{C^{0,\alpha}(B_{2R}(x))}\leq C(n,R)$, where $C(n,R)$ is a constant  independent of $m$. Hence $\|u_m\|_{C^{0,\alpha}(B_{2R}(x))}\leq C(n,R)$. Local Schauder estimate and a bootstrip argument yields that $\|u_m\|_{C^{2,\alpha^\prime}(B_R(x))}\leq C(n,R)$. Thanks to Azell\`{a}-Ascolli theorem, we have $u_m$ convergent uniformly to a $C^{2,\alpha^\prime}_{loc}$ function $u_\infty=W_\infty+\phi_\infty$ on compact sets as $m\to \infty$. We know $u_\infty$ satisfies $|u_\infty|_{L^\infty(\mathbb{R}^n)}\leq C$ and $\|u_\infty\|_{C^{2,\alpha^\prime}(B_R(x))}\leq C(n,R)$.

We will show that $u_\infty$ satisfies the equation \eqref{2}. Let $v_m=u_m-u_\infty$. From above, we know $v_m$ has the property $|v_m|_{L^\infty(\mathbb{R}^n)}\leq C$; $|v_m|_{C^2(B_1(x))}\leq C$ and $v_m\to 0$ uniformly on compact sets. From the definition of $(-\Delta)^s$, we have for any $x\in \mathbb{R}^n$
\begin{eqnarray*}
&&C(n,s)^{-1}|(-\Delta)^s v_m(x)| \\
&\leq &P.V\int_{\mathbb{R}^n}\frac{|v_m(x)-v_m(y)|}{|x-y|^{n+2s}}dy \\
&=&P.V.\int_{B_{\varepsilon_0}(x)}\frac{|v_m(x)-v_m(y)|}{|x-y|^{n+2s}}dy+\int_{B_R(x)\char92 B_{\varepsilon_0}(x)}\frac{|v_m(x)-v_m(y)|}{ |x-y|^{n+2s}}dy \\
 &&+\int_{\mathbb{R}^n\char92 B_{R}(x)}\frac{|v_m(x)-v_m(y)|}{|x-y|^{n+2s}}dy \\
&=:& T_1+T_2+T_3.
\end{eqnarray*}
For the term $T_1$, we have
\begin{eqnarray*}
T_1&=& \frac12P.V.\int_{B_{\varepsilon_0}(0)}\frac{|v_m(x+y)+v_m(x-y)-2v_m(x)|}{|y|^{n+2s}} \\
&\leq& C|v_m|_{C^2{B_1(x)}} \int_{B_{\varepsilon_0}(0)}|y|^{2-2s-n}dy\leq C|v_m|_{C^2{B_1(x)}}\varepsilon_0^{2-2s}\to 0 \text{ as } \varepsilon_0\to 0.
\end{eqnarray*}
For the third term,
\begin{equation*}
T_3\leq C\int_{\mathbb{R}^n\char92 B_R(x)}\frac1{|x-y|^{n+2s}}=CR^{-2s}\to 0 \text{ as } R\to \infty.
\end{equation*}
Then we estimate the term $T_2$. For fixed $R$ large enough and $\varepsilon_0$ small enough, $B_R(x)\char92 B_{\varepsilon_0}(x)$ is a compact set. So we have $T_2\to 0$ as $m\to\infty$. Hence $(-\Delta)^s u_m(x) \to (-\Delta)^s u_\infty(x)$ as $m\to\infty$. Therefore $u_\infty$ satisfies equation \eqref{2}.

\hfill{\usefont{U}{msa}{m}{n}\char"03}

\noindent\textit{Proof of Corollary  \ref{coro1}.}
Fix the constant $m<\infty$. Using a similar argument as in \cite{WeiYan2010}, we can expand $I(W_m+\phi_m)$ as
\[I(W_m+\phi_m)=(m+1)^k\left(\frac sn\int U_{0,1}^{\frac{2n}{n-2s}}+o(1)\right),\text{ as }l\to\infty.\]
For each $m<\infty$, $I(W_m+\phi_m)\to (m+1)^k\frac sn\int U_{0,1}^{\frac{2n}{n-2s}}$ as $l\to\infty$. For any $m_1,m_2\in \mathbb{N}_+$ such that $m_1\not=m_2$, we can find two solutions $W_{m_1}+\phi_{m_1}$ and $W_{m_2}+\phi_{m_2}$ of \eqref{2}, such that $I(W_{m_1}+\phi_{m_1})\not=I(W_{m_2}+\phi_{m_2})$. Hence we can find infinitely many solutions of \eqref{2}.

\hfill{\usefont{U}{msa}{m}{n}\char"03}

\appendix
\section{Basic Estimates}\label{sect4}
\begin{lemma}\label{lm4}
(cf. \cite{LiYY2016,WeiYan2010}) For any $x_i,x_j,y\in \mathbb{R}^n$ and constant $\sigma\in[0,\min\{\alpha,\beta\}]$, we have
\[\frac1{(1+|y-x_i|)^\alpha(1+|y-x_j|)^\beta}\leq \frac {2^\sigma}{(1+|x_i-x_j|)^{\sigma}}\left(\frac1{(1+|y-x_i|)^{\alpha+\beta-\sigma}}+ \frac1{(1+|y-x_j|)^{\alpha+\beta-\sigma}}\right).\]
\end{lemma}

\begin{lemma}\label{lm10}
For any $\sigma>0$  with $\sigma\not=n-2s$, there is a constant $C>0$ such that
\[\int_{\mathbb{R}^n}\frac{1}{|y-z|^{n-2s}}\frac{1}{(1+|z|)^{2s+\sigma}}dz\leq\frac{C}{(1+|y| )^{\min(\sigma,n-2s)}} .\]
For $\sigma=n-2s$, there is also a constant $C>0$, such that
\[\int_{\mathbb{R}^n}\frac1{|y-z|^{n-2s}}\frac1{(1+|z|)^n}\leq C\frac{\max(1,\log |y|)}{(1+|y|)^{n-2s}}\]
\end{lemma}

\begin{proof}
The proof follows from the same argument as \cite[Lemma A.2]{LiYY2016}. See also  \cite[Lemma B.2]{WeiYan2010}.

\end{proof}
Recall that $X^i\in X_{i,m}=\{X^i\}_{i=1}^{(m+1)^k}$, $B_i=B_{\lambda l}(X^i)$ and $B_{i,m}=B_{\max\{\frac m4,1\}\lambda l}(X^i)$.
\begin{lemma}\label{lm9}
(cf. \cite{LiYY2016}) For any $\theta>k$, there exists a constant $C(\theta,k,n)>1$ independent of $m$, such that if $y\in B_i\cap \Omega_i$, there holds
\begin{equation*}
\frac1{(1+|y-X^i|)^\theta}\leq \sum_j\frac1{(1+|y-X^j|)^\theta}\leq \frac C{(1+|y-X^i|)^\theta}.
\end{equation*}
If $y\in B_i^c\cap B_{i,m}\cap \Omega_i$, there holds
\begin{equation}\label{19}
\frac 1 {C(1+|y-X^i|)^{\theta-k}(\lambda l)^k}\leq \sum_j\frac1{(1+|y-X^j|)^\theta}\leq \frac C {(1+|y-X^i|)^{\theta-k}(\lambda l)^k}.
\end{equation}
and if $y\in B_{i,m}^c\cap \Omega_i$, there holds
\begin{equation*}
\frac {m^k}{C(1+|y-X^i|)^\theta}\leq \sum_j\frac1{(1+|y-X^j|)^\theta} \leq\frac {Cm^k}{(1+|y-X^i|)^{ \theta}}\leq \frac C {(1+|y-X^i|)^{\theta-k}(\lambda l)^k}.
\end{equation*}
\end{lemma}

\begin{lemma}\label{lm13}
 Let $n>2s+2$ and $0<\tau<\frac{n+2s}2$. If  $\phi$ satisfies $\|\phi\|_*\leq \frac C{\lambda^{\frac{n+2s}2-\tau}}$, then for any $c>0$, there exists $\lambda_0$ such that for any $\lambda>\lambda_0$, there holds $|\phi|\leq cW_m$ in $\cup_h(\Omega_h\cap B_h)$.
\end{lemma}
\begin{proof}
We prove this lemma indirectly. Suppose that there exists $c_0>0$, such that for any $\lambda_0>0$, there is a $\lambda>\lambda_0$ and $y\in \cup_l(\Omega_l\cap B_l)$ such that $|\phi(y)|\geq c_0 W_m(y)$. Then
\begin{eqnarray*}
|\phi(y)|&\geq& C\sum_{h=1}^{(m+1)^k}\frac1{(1+|y-X^h|)^{n-2s}} \\
&\geq& C\gamma(y)\sum_{h=1}^{(m+1)^k} \frac1{(1+|y-X^h|)^{ \frac{n-2s}2+\tau}}\frac1{(\lambda l)^{\frac{n-2s}2-\tau}}
\end{eqnarray*}
If $0<\tau<\frac{n-2s}2$, we have
\begin{equation*}
\frac1{\lambda^{\frac{n+2s}2-\tau}}\geq\|\phi\|_*\geq\frac C{(\lambda l)^{\frac{n-2s}2-\tau}},
\end{equation*}
which does not hold for $\lambda$ large enough.

If $\frac{n-2s}2\leq\tau<\frac{n+2s}2$, we can also get
\begin{equation*}
\frac1{\lambda^{\frac{n+2s}2-\tau}}\geq\|\phi\|_*\geq C,
\end{equation*}
which also is a contradiction for $\lambda$ large.

\end{proof}

\begin{lemma}\label{lm11}
For $n>2s+2$ and $1\leq k<\tau<\frac{n-2s}2$, we have
\begin{eqnarray*}
&&\int_{\mathbb{R}^n}\frac1{|x-y|^{n-2s}}W_{P, \Lambda}^{\frac{4s}{n-2s}}(y)\gamma(y)\sum_h\frac1{( 1+|y-X^h|)^{\frac{n-2s}2+\tau}}dy \\
&\leq&C\left(\gamma(x)\sum_h\frac1{( 1+|x-X^h|)^{\frac{n-2s}2+\tau+\theta}}+\frac1{(\lambda l)^{\frac{4s}{n-2s}k}}\gamma(x)\sum_h\frac1{( 1+|x-X^h|)^{\frac{n-2s}2+\tau}}\right),
\end{eqnarray*}
where $\theta>0$ is a small constant and $C>0$ does not depend on $m$.
\end{lemma}

\begin{proof}
Without loss of generality, we assume $x\in \Omega_1$. We write
\begin{eqnarray*}
&&\int_{\mathbb{R}^n}\frac1{|x-y|^{n-2s}}W_m^{\frac{4s}{n-2s}}(y)\gamma(y)\sum_h\frac1{( 1+|y-X^h|)^{\frac{n-2s}2+\tau}}dy \\
&=&\left(\int_{\cup_h(\Omega_h\cap B_h)}+\int_{\cup_h(\Omega_h\cap B_h^c \cap B_{h,m})}+\int_{\cup_h(\Omega_h\cap B_{h,m}^c)}\right)\frac1{|x-y|^{n-2s}} W_m^{\frac{4s}{n-2s}}(y) \\
&&\;\;\;\;\;\;\times\gamma(y)\sum_h\frac1{( 1+|y-X^h|)^{\frac{n-2s}2+\tau}}dy \\
&=:&T_1+T_2+T_3.
\end{eqnarray*}
We now estimate each term $T_i(\;\;i=1,2,3)$.

Using Lemma \ref{lm10} and Lemma \ref{lm9}, we have
\begin{eqnarray}\label{61}
\nonumber T_1&\leq& C\int_{\cup_h(\Omega_h\cap B_h)}\frac1{|x-y|^{n-2s}}\sum_h\frac1{(1+|y-X^h|)^{4s+\frac{n-2s}2+\tau}}dy \\\nonumber
&\leq& C\sum_h\frac1{(1+|x-X^h|)^{\min\{\frac{n+2s}2+\tau,n-2s\}}} \\
&=& C\sum_h\frac1{(1+|x-X^h|)^{\frac{n-2s}2+\tau+\theta_1}},\hspace{1,cm}\text{where }\theta_1=\min\{2s,\frac{n-2s}2-\tau\}.
\end{eqnarray}
Similarly, we also obtain
\begin{eqnarray}\label{62}
\nonumber T_1&\leq& C\int_{\cup_h(\Omega_h\cap B_h)}\frac1{|x-y|^{n-2s}}\frac1{\lambda^{\tau-s}}\sum_h\frac1{(1+|y-X^h| )^{\frac n2+4s}}dy \\
\nonumber&\leq& \frac C{\lambda^{\tau-s}}\sum_h\frac1{(1+|x-X^h|)^{\min\{\frac n2+2s,n-2s\}}} \\ \nonumber
&\leq& C\frac{(1+|x-X^1|)^{\tau-s}}{\lambda^{\tau-s}}\sum_h\frac1{(1+|x-X^h|)^{\min\{n-3s+\tau,
\frac n2+\tau+s\}}} \\
&=& C\frac{(1+|x-X^1|)^{\tau-s}}{\lambda^{\tau-s}}\sum_h\frac1{(1+|x-X^h|)^{\frac{n-2s}2+\tau+\theta_2}},
\end{eqnarray}
where $\theta_2=\min\{\frac{n-4s}2,2s\}$.

Combining the estimate \eqref{61} and \eqref{62}, we have
\[T_1\leq C\gamma(x)\sum_h\frac1{(1+|x-X^h|)^{\frac{n-2s}2+\tau+\theta}},\]
where $\theta=\min\{2s,\frac{n-4s}2,\frac{n-2s}2-\tau\}$.

For the term $T_2$, we have
\begin{eqnarray*}
T_2&\leq& \frac C{(\lambda l)^{\frac{4s}{n-2s}k}}\int_{\mathbb{R}^n}\frac1{|x-y|^{n-2s}}\sum_h\frac1{(1+ |y-X^h|)^{\frac{n-2s}2+4s+\tau-\frac{4s}{n-2s}k}}dy \\
&\leq& \frac C{(\lambda l)^{\frac{4s}{n-2s}k}}\sum_h\frac1{(1+|x-X^h|)^{\min\{n-2s,\frac{n-2s}2+\tau+2s- \frac{4s}{n-2s}k\}}} \\
&\leq& \frac C{(\lambda l)^{\frac{4s}{n-2s}k}}\sum_h\frac1{(1+|x-X^h|)^{\frac{n-2s}2+\tau}},
\end{eqnarray*}
and
\begin{eqnarray*}
T_2&\leq&\frac C{\lambda^{\tau-s}(\lambda l)^{\frac{4s}{n-2s}k}}\int_{\mathbb{R}^n}\frac1{|x-y|^{n-2s}} \sum_h\frac1{(1+|y-X^h|)^{\frac n2+4s-\frac{4s}{n-2s}k}}dy \\
&\leq& \frac C{\lambda^{\tau-s}(\lambda l)^{\frac{4s}{n-2s}k}}\sum_h\frac1{(1+|x-X^h|)^{\min\{n-2s, \frac n2+2s-\frac{4s}{n-2s}k\}}} \\
&\leq& \frac C{(\lambda l)^{\frac{4s}{n-2s}k}}\frac{(1+|x-X^1|)^{\tau-s}}{\lambda^{\tau-s}}\sum_h\frac1{(1+ |x-X^h|)^{\frac{n-2s}2+\tau}}.
\end{eqnarray*}
Thus
\[T_2\leq \frac C{(\lambda l)^{\frac{4s}{n-2s}k}}\gamma(x)\sum_h\frac1{(1+ |x-X^h|)^{\frac{n-2s}2+\tau}}.\]
By the same procedure, we have
\[T_3\leq \frac C{(\lambda l)^{\frac{4s}{n-2s}k}}\gamma(x)\sum_h\frac1{(1+ |x-X^h|)^{\frac{n-2s}2+\tau}}.\]
Hence this lemma follows.

\end{proof}
Remember   $Z_{i,j}=\frac{\partial U_{P^i,\Lambda_i}}{\partial P^i_j}$ for $j=1,2,\dots,n$ and $Z_{i,n+1}=\frac{\partial U_{P^i,\Lambda_i}}{\partial \Lambda_i}$.

\begin{lemma}\label{lm5}
For $t=1,2\dots,n+1$, we have
\begin{equation*}
\left|\int_{\mathbb{R}^n}K(\frac x\lambda)W_m^{\frac{4s}{n-2s}}Z_{r,t}\phi\right|\leq \frac{C\|\phi\|_*}{\lambda^{\tau-s}(\lambda l)^{\frac n2}}.
\end{equation*}
\end{lemma}

\begin{proof}
In the proof of this lemma, we denote $\hat{W}_{m,r}=\sum_{h\not=r} U_{P^h,\Lambda_h}$. It is easy to get
\begin{eqnarray}\label{20}
\int_{\mathbb{R}^n}K(\frac x\lambda)W_m^{\frac{4s}{n-2s}}Z_{r,t}\phi&=&\int_{\mathbb{R}^n}K(\frac x\lambda)U_{P^r,\Lambda_r}^{\frac{4s}{n-2s}}Z_{r,t}\phi+O\left(\int_{\hat{W}_{m,r}>U_{P^r,\Lambda_r}} \hat{W}_{m,r}^{\frac{4s}{n-2s}}Z_{r,t}\phi\right) \\
\nonumber&&+O\left(\int_{\hat{W}_{m,r}\leq U_{P^r,\Lambda_r}} U_{P^r,\Lambda_r}^{\frac{4s}{n-2s}}\hat{W}_{m,r}\phi\right).
\end{eqnarray}
We need to estimate each term in the equality above.

For $i\not=r$, from Lemma \ref{lm9}, we have
\begin{eqnarray*}
&&\left|\int_{\Omega_i\cap B_i}\hat{W}_{m,r}^{\frac{4s}{n-2s}}Z_{r,t}\phi\right| \\
&\leq& C\frac{\|\phi\|_*}{\lambda^{\tau-s}}\int_{\Omega_i\cap B_i}\left(\sum_{h\not=r}\frac1{(1+|x-X^h|)^{n-2s}} \right)^{\frac{4s}{n-2s}}\frac1{(1+|x-X^r|)^{n-2s}}\sum_h\frac1{(1+|x-X^h|)^{\frac n2}}dx \\
&\leq& C\frac{\|\phi\|_*}{\lambda^{\tau-s}}\int_{\Omega_i\cap B_i}\frac1{(1+|x-X^r|)^{n-2s}}\frac1{(1+|x-X^i|)^{\frac n2+4s}}dx\leq  C\frac{\|\phi\|_*}{\lambda^{\tau-s}|X^r-X^i|^{\frac n2}}.
\end{eqnarray*}
With the help of Lemma \ref{lm4} and Lemma \ref{lm9}, we get
\begin{eqnarray}\label{43}
\nonumber&&\left|\int_{\cup_h(\Omega_h\cap B_h^c)}\hat{W}_{m,r}^{\frac{4s}{n-2s}}Z_{r,t}\phi\right|  \\ \nonumber
&\leq& C\|\phi\|_*\int_{\cup_h(\Omega_h\cap B_h^c)}\frac1{(1+|x-X^r|)^{n-2s}}\left(\sum_{h\not=r}\frac1{(1+|x-X^h|)^{n-2s}}\right)^{\frac{4s}{n-2s}} \\
&&\hspace{2,cm}\times\left(\sum_h\frac1{(1+|x-X^h|)^{\frac{n-2s}2+\tau}}\right)dx \\ \nonumber
&\leq& C\frac{\|\phi\|_*}{(\lambda l)^{\frac{4s}{n-2s}k}}\int_{\cup_h(\Omega_h\cap B_h^c)}\frac1{(1+|x-X^r|)^{n-2s}}\sum_h\frac1{(1+|x-X^h|)^{\frac{n}2+3s+\tau-\frac{4s}{n-2s}k}}dx \\
&\leq& C\frac{\|\phi\|_*}{(\lambda l)^{\frac{n}2+s+\tau}}.\nonumber
\end{eqnarray}

Since in $\Omega_r\cap B_r$, there holds $\hat{W}_{m,r}\leq \sum_{j\not=r}\frac C{|X^j-X^r|^{n-2s}}\leq \frac C{(\lambda l)^{n-2s}}$. Then if $n\geq 6s$, we have
\begin{eqnarray*}
&&\left|\int_{\Omega_r\cap B_r\atop \hat{W}_{m,r}>U_{P^r,\Lambda_r}} \hat{W}_{m,r}^{\frac{4s}{n-2s}}Z_{r,t}\phi\right|\leq C\int_{\Omega_r\cap B_r}\hat{W}_{m,r}^{\frac{n+2s}{2(n-2s)}}U_{P^r,\Lambda_r}^{\frac{n+2s}{2(n-2s)}}|\phi| \\
&\leq& C\frac{\|\phi\|_*}{\lambda^{\tau-s}}\frac1{(\lambda l)^{\frac{n+2s}2}}\int_{\Omega_r\cap B_r} \frac1{(1+|x-X^r|)^{\frac{n+2s}2}}\sum_h\frac1{(1+|x-X^h|)^{\frac n2}}dx \\
&\leq& C\frac{\|\phi\|_*}{\lambda^{\tau-s}(\lambda l)^{\frac{n+2s}2}}.
\end{eqnarray*}
And if $n<6s$, we get $\frac{n+2s}2<4s$. In this case
\begin{eqnarray*}
\left|\int_{\Omega_r\cap B_r\atop \hat{W}_{m,r}>U_{P^r,\Lambda_r}} \hat{W}_{m,r}^{\frac{4s}{n-2s}}Z_{r,t}\phi\right|&\leq& \frac{C\|\phi\|_*}{\lambda^{\tau-s}(\lambda l)^{4s}}\int_{\Omega_r\cap B_r}\frac1{(1+|x-X^r|)^{n-2s}}\sum_h\frac1{(1+|x-X^h|)^{\frac n2}}dx \\
&\leq& \frac{C\|\phi\|_*}{\lambda^{\tau-s}(\lambda l)^{\frac{n+2s}2}}.
\end{eqnarray*}
From these arguments above, we arrive
\begin{equation}\label{48}
\left|\int_{\hat{W}_{m,r}>U_{P^r,\Lambda_r}} \hat{W}_{m,r}^{\frac{4s}{n-2s}}Z_{r,t}\phi\right|\leq C\frac{\|\phi\|_*}{\lambda^{\tau-s}(\lambda l)^{\frac n2}}.
\end{equation}

By a similar procedure, we get
\begin{equation}\label{49}
\left|\int_{\hat{W}_{m,r}\leq U_{P^r,\Lambda_r}}U_{P^r,\Lambda_r}^{\frac {4s}{n-2s}}\hat{W}_{m,r}\phi\right|\leq C\frac{\|\phi\|_*}{\lambda^{\tau-s}(\lambda l)^{\frac{n}2}}.
\end{equation}

Now we estimate the first term on the right hand side of the equality \eqref{20}. Since $\phi$ satisfies the second equality in \eqref{17}, we have
\begin{eqnarray*}
&&\left|\int_{\mathbb{R}^n}K(\frac x \lambda)U_{P^r,\Lambda_r}^{\frac{4s}{n-2s}}Z_{r,t}\phi\right|= \left|\int_{\mathbb{R}^n}\left(K(\frac x \lambda)-1\right)U_{P^r,\Lambda_r}^{\frac{4s}{n-2s}}Z_{r,t}\phi\right| \\
& \leq& \frac{\|\phi\|_*}{\lambda^{\tau-s}}\int_{\mathbb{R}^n}\left|K(\frac x \lambda)-1\right|U_{P^r,\Lambda_r}^{\frac{n+2s}{n-2s}}\sum_h\frac1{(1+|x-X^h|)^{\frac{n}2}}dx.
\end{eqnarray*}
On one hand, Lemma \ref{lm4} implies that
\begin{eqnarray*}
&&\int_{\mathbb{R}^n}\left|K(\frac x \lambda)-1\right|U_{P^r,\Lambda_r}^{\frac{n+2s}{n-2s}}\sum_{h\not=r}\frac1{(1+|x-X^h|)^{\frac{n}2}}dx \\
&\leq& C\int_{\mathbb{R}^n}\frac1{(1+|x-X^r|)^{n+2s}}\sum_{h\not=r}\frac1{(1+|x-X^h|)^{\frac{n}2}}dx \\
&\leq& \frac{C}{(\lambda l)^{\frac{n}2}}.
\end{eqnarray*}
On the another hand, choose $\delta$ to be a fixed constant small enough,
\begin{eqnarray*}
&&\int_{\mathbb{R}^n}\left|K(\frac x \lambda)-1\right|U_{P^r,\Lambda_r}^{\frac{n+2s}{n-2s}}\frac1{(1+|x-X^r|)^{\frac{n}2}}\\
&=& \int_{|x-X^r|\leq \delta\lambda}\left|K(\frac x \lambda)-1\right|U_{P^r,\Lambda_r}^{\frac{n+2s}{n-2s}}\frac1{(1+|x-X^r|)^{\frac{n}2}} \\
&&+ \int_{|x-X^r|>\delta\lambda}\left|K(\frac x \lambda)-1\right|U_{P^r,\Lambda_r}^{\frac{n+2s}{n-2s}}\frac1{(1+|x-X^r|)^{\frac{n}2}}=:J_1+J_2.
\end{eqnarray*}
From the condition ($H_3$), we have
\begin{equation*}
|J_1|\leq \frac C{\lambda^\beta}\int_{|x-X^r|\leq \delta \lambda}\frac{|x-X^r|^\beta}{(1+|x-X^r|)^{n+2s+\frac{n}2}}\leq\left\{\begin{array}{lll}
\frac {C\log \lambda}{\lambda^{\frac{n+4s}2}},& \mbox{if $\beta\geq\frac{n+4s}2$,}\\
\\
\frac C{\lambda^\beta}, &\mbox{if $\beta<\frac{n+4s}2$.}\end{array}
\right.
\end{equation*}
For the term $J_2$, a direct calculation yields
\begin{eqnarray*}
J_2&=&\int_{|x-X^r|> \delta\lambda} \left|K(\frac x \lambda)-1\right|U_{P^r,\Lambda_r}^{\frac{n+2s}{n-2s}}\frac1{(1+|x-X^r|)^{\frac{n}2}} \\
&\leq& \int_{|x-X^r|> \delta\lambda}\frac1{(1+|x-X^r|)^{n+2s+\frac{n}2}} \\
&\leq& \frac C{\lambda^{\frac{n+4s}{2}}}.
\end{eqnarray*}
Since $\min\{\beta,\frac{n+4s}2\}>\frac n2\frac \beta{n-2s}$, the definition of $\lambda$ implies
\[\int_{\mathbb{R}^n}\left|K(\frac x \lambda)-1\right|U_{P^r,\Lambda_r}^{\frac{n+2s}{n-2s}}\frac1{(1+|x-X^r|)^{\frac{n}2}}\leq \frac C{(\lambda l)^{\frac n2}}.\]
Hence we obtain
\begin{equation}\label{50}
\left|\int_{\mathbb{R}^n}K(\frac x \lambda)U_{P^r,\Lambda_r}^{\frac{4s}{n-2s}}Z_{r,t}\phi\right|\leq C\frac{\|\phi\|_*}{\lambda^{\tau-s}(\lambda l)^{\frac{n}2}}.
\end{equation}
Putting \eqref{48}, \eqref{49} and \eqref{50} into \eqref{20}, we get this lemma.

\end{proof}

\begin{lemma}\label{lm14}
It holds that
\begin{equation}\label{42}
\int_{\cup_h (\Omega_h\cap B_h^c)}|\phi|^{\frac{n+2s}{n-2s}}U_{P^i,\Lambda_i}\leq C\frac{\|\phi\|_*^{\frac{n+2s}{n-2s}}}{(\lambda l)^{\frac{n-2s}2+\frac{n+2s}{n-2s}\tau}},
\end{equation}
and
\begin{equation}\label{44}
\left|\int_{\mathbb{R}^n}W_m^{\frac{6s-n}{n-2s}}\phi^2U_{P^i,\Lambda_i}\right|\leq \frac{C\|\phi\|_*^2}{\lambda^{2(\tau-s)}}.
\end{equation}
\end{lemma}

\begin{proof}
The estimate \eqref{42} follows by the same method as in \eqref{43}. So we only prove the estimation \eqref{44}.

Using the same trick as in \eqref{43}, we have
\begin{equation}\label{45}
\left|\int_{\cup_h (\Omega_h\cap B_h^c)}W_m^{\frac{6s-n}{n-2s}}\phi^2U_{P^i,\Lambda_i}\right|\leq \frac{C\|\phi\|_*^2}{\lambda^{2(\tau+s)}}.
\end{equation}
According to Lemma \ref{lm9}, we have for $t\not=i$,
\begin{equation*}
\left|\int_{\Omega_t\cap B_t}W_m^{\frac{6s-n}{n-2s}}\phi^2U_{P^i,\Lambda_i}\right|\leq \frac{C\|\phi\|_*^2}{(\lambda l)^{2(\tau-s)}}\int_{\Omega_t\cap B_t}\frac1{(1+|y-X^t|)^{6s}}\frac1{(1+|y-X^i|)^{n-2s}}dy.
\end{equation*}
If $n\geq 6s$
\[\int_{\Omega_t\cap B_t}\frac1{(1+|y-X^t|)^{6s}}\frac1{(1+|y-X^i|)^{n-2s}}dy\leq \frac{(\lambda l)^{n-6s}\log(\lambda l)}{|X^i-X^t|^{n-2s}};\]
If otherwise,  $n< 6s$, there holds
\[\int_{\Omega_t\cap B_t}\frac1{(1+|y-X^t|)^{6s}}\frac1{(1+|y-X^i|)^{n-2s}}dy \leq \frac1{|X^i-X^t|^{n-2s}}.\]
Hence
\begin{equation}\label{46}
\sum_{t\not=i}\left|\int_{\Omega_t\cap B_t}W_m^{\frac{6s-n}{n-2s}}\phi^2U_{P^i,\Lambda_i}\right|\leq \frac{C\|\phi\|_*^2}{(\lambda l)^{2(\tau-s)}}\frac {C\log(\lambda l)}{(\lambda l)^{\min\{n-2s,4s\}}}.
\end{equation}
Using Lemma \ref{lm9}, we also have
\begin{equation}\label{47}
\left|\int_{\Omega_i\cap B_i}W_m^{\frac{6s-n}{n-2s}}\phi^2U_{P^i,\Lambda_i}\right|\leq \frac{C\|\phi\|_*^2}{(\lambda l)^{2(\tau-s)}}\int_{\Omega_i\cap B_i}\frac1{(1+|y-X^i|)^{n+4s}}.
\end{equation}
So we obtain the estimate \eqref{44} from \eqref{45}, \eqref{46} and \eqref{47}.

\end{proof}

\begin{lemma}\label{lm12}
If $\phi$ is the solution of the equation
\begin{equation}\label{29}
(-\Delta)^s\phi(x)-\frac{n+2s}{n-2s}K(\frac x{\lambda})W_m^{\frac{4s}{n-2s}}(x)\phi(x)=g(x),
\end{equation}
satisfying $\|\phi\|_*<+\infty$, then we have \[\displaystyle\sup_{x_1\not=x_2}\frac{|\lambda^{\tau-s}\phi(x_1)-\lambda^{\tau-s}\phi(x_2)|}{|x_1-x_2 |^\alpha}\leq C\max\{\|\phi\|_*,\|g\|_{**}\},\]
where $\alpha=\min\{2s,1\}$ and the constant $C$ does not depend on $\lambda$ and $m$.
\end{lemma}

\begin{proof}
Since $\displaystyle |\lambda^{\tau-s}\phi(x)|\leq \|\phi\|_*\sum_h\frac1{(1+|x-X^h|)^{\frac n2}}\leq C\|\phi\|_*,$ we can assume $|x_1-x_2|\leq\frac13$ with no loss of generality.
Using the Green function of $(-\Delta)^s$( see \cite{Caffarelli2007}), we can write \eqref{29} into the following form
\[\phi(x)=C\int_{\mathbb{R}^n}\frac1{|x-y|^{n-2s}}\left(\frac{n+2s}{n-2s}K(\frac y {\lambda})W_m^{ \frac{4s}{n-2s} }(y) \phi(y)+g(y)\right)dy,\]
Then we get
\begin{eqnarray*}
|\phi(x_1)-\phi(x_2)|&\leq& C\left|\int_{\mathbb{R}^n}\left(\frac1{|x_1-y|^{n-2s}}-\frac1{|x_2-y|^{n-2s}}\right) K(\frac y\lambda)W_m^{\frac{4s}{n-2s}}(y)\phi(y)dy\right| \\
&& + C\left|\int_{\mathbb{R}^n}\left(\frac1{|x_1-y|^{n-2s}}-\frac1{ |x_2-y|^{n-2s}}\right)g(y)dy\right| \\
&=:& C(H_1+H_2).
\end{eqnarray*}
Using the definition of the norm $\|\cdot\|_*$, there hold
\begin{eqnarray*}
&&|H_1|\leq \|\phi\|_*\left|\int_{\mathbb{R}^n}\left(\frac1{|x_1-x_2-y|^{n-2s}}-\frac1{|y|^{n-2s}}\right) K(\frac{y+x_2}\lambda)W_m^{\frac{4s}{n-2s}}(y+x_2)\times\right. \\
&&\hspace{1,cm}\left.\times\gamma(y+x_2)\sum_h\frac1{(1+|y+x_2-X^h|)^{ \frac{n-2s}2+\tau}}dy\right| \\
&=&\|\phi\|_*\left\{\int_{|y|\leq 3|x_1-x_2|}\left(\frac1{|x_1-x_2-y|^{n-2s}}-\frac1{|y|^{n-2s}}\right) K(\frac{y+x_2}\lambda)W_m^{ \frac{ 4s}{n-2s}}(y+x_2)\times\right. \\
&&\hspace{1,cm}\times\gamma(y+x_2)\sum_h\frac1{(1+|y+x_2-X^h|)^{ \frac{n-2s}2+\tau}}dy \\
&&+\int_{|y|\geq 3|x_1-x_2|} \left(\frac1{|x_1-x_2-y|^{n-2s}}-\frac1{|y|^{n-2s}}\right) K(\frac{y+x_2}\lambda)W_m^{ \frac{ 4s}{n-2s}}(y+x_2)\times \\
&&\left.\hspace{1,cm}\times\gamma(y+x_2)\sum_h\frac1{(1+|y+x_2-X^h|)^{ \frac{n-2s}2+\tau}}dy\right\} \\
&=:& \|\phi\|_*(K_1+K_2).
\end{eqnarray*}
For the term $K_1$, we have
\begin{equation*}
|K_1|\leq \frac C{\lambda^{\tau-s}}\int_{|y|\leq 4|x_1-x_2|}\frac1{|y|^{n-2s}}\leq \frac C{\lambda^{\tau-s}}|x_1-x_2|^{2s}.
\end{equation*}
For the term $K_2$, we have
\begin{eqnarray*}
|K_2|&\leq& C|x_1-x_2|\int_0^1dt\int_{|y|\geq 3|x_1-x_2|}\frac1{|t(x_1-x_2)-y|^{n-2s+1}}W_m^{\frac{ 4s}{n-2s}}(y+x_2) \\
&&\hspace{1,cm}\times\gamma(y+x_2)\sum_h\frac1{(1+|y+x_2-X^h|)^{\frac{n-2s}2+\tau}}dy \\
&=& C|x_1-x_2| \int_0^1dt\left(\int_{1>|y|\geq 3|x_1-x_2|}+\int_{|y|\geq 1}\right)\frac1{|t(x_1-x_2)-y|^{n-2s+1}}W_m^{\frac{4s}{n-2s}}(y+x_2) \\
&&\hspace{1,cm}\times\gamma(y+x_2)\sum_h\frac1{(1+|y+x_2-X^h|)^{\frac{n-2s}2+\tau}}dy \\
&=:& C|x_1-x_2|(M_1+M_2).
\end{eqnarray*}
Since it holds that $|t(x_1-x_2)-y|\in [2|x_1-x_2|,\frac43)$ for $1>|y|\geq 3|x_1-x_2|$, we get
\begin{equation*}
|M_1|\leq \frac1{\lambda^{\tau-s}}\int_{2|x_1-x_2|\leq |y|\leq \frac43}\frac1{|y|^{n-2s+1}}\leq \frac1{\lambda^{\tau-s}} (C+C|x_1-x_2|^{2s-1}).
\end{equation*}
For $|y|\geq1$, we have $\frac23|y|\leq |t(x_1-x_2)-y|\leq\frac43|y|$. Lemma \ref{lm11} yields
\begin{eqnarray*}
M_2&\leq& C\int_{|y|\geq 1}\frac1{|y|^{n-2s+1}}W_m^{\frac{4s}{n-2s}}(y+x_2)\gamma(y+x_2) \sum_h\frac1{(1+|y+x_2-X^h|)^{\frac{n-2s}2+\tau}}dy \\
&\leq& C\int_{\mathbb{R}^n}\frac1{|y-x_2|^{n-2s}}W_m^{\frac{4s}{n-2s}}(y) \gamma(y)\sum_h\frac1{(1+|y-X^h|)^{\frac{n-2s}2+\tau}}dy \\
&\leq& \frac C{\lambda^{\tau-s}}.
\end{eqnarray*}
Hence $|H_1|\leq \frac C{\lambda^{\tau-s}}\|\phi\|_*|x_1-x_2|^\alpha$. The same procedure with the help of Lemma \ref{lm10} yields that
$|H_2|\leq \frac C{\lambda^{\tau-s}}\|g\|_{**}|x_1-x_2|^\alpha$. Then Lemma \ref{lm12} follows.

\end{proof}

\section{Expansions of the functionals $\frac \partial {\partial \Lambda_i}I(W_m)$ and $\frac \partial {\partial P^i_{ j}}I(W_m)$}\label{sect5}
In this section, we will expand the functionals $\frac \partial {\partial \Lambda_i}I(W_m)$ and $\frac \partial {\partial P^i_{ j}}I(W_m)$. A direct computation yields
\begin{eqnarray}\label{8}
\nonumber \frac {\partial I} {\partial \Lambda_i}(W_m)&=&\int_{\mathbb{R}^n} W_m(-\Delta)^s \frac{\partial W_m}{\partial \Lambda_i}-\int_{\mathbb{R}^n}K\left(\frac x\lambda\right)W_m^{\frac{n+2s}{n-2s}}\frac{\partial W_m}{\partial \Lambda_i} \\
&=& \int_{\mathbb{R}^n}\sum_{h=1}^{(m+1)^k} U_{P^h,\Lambda_h}^{\frac{n+2s}{n-2s}}\frac{\partial U_{P^i,\Lambda_i}}{\partial \Lambda_i}-\int_{\mathbb{R}^n}K\left(\frac x \lambda\right)W_m^{\frac{n+2s}{n-2s}}\frac{\partial U_{P^i,\Lambda_i}}{\partial \Lambda_i},
\end{eqnarray}
and
\begin{eqnarray}\label{7}
\nonumber \frac {\partial I} {\partial P^i_j}(W_m)&=&\int_{\mathbb{R}^n} W_m(-\Delta)^s \frac{\partial W_m}{\partial P^i_j}-\int_{\mathbb{R}^n}K\left(\frac x\lambda\right)W_m^{\frac{n+2s}{n-2s}}\frac{\partial W_m}{\partial P^i_j} \\
&=& \int_{\mathbb{R}^n}\sum_{h\not=i} U_{P^h,\Lambda_h}^{\frac{n+2s}{n-2s}}\frac{\partial U_{P^i,\Lambda_i}}{\partial P^i_j}-\int_{\mathbb{R}^n}K\left(\frac x \lambda\right)W_m^{\frac{n+2s}{n-2s}}\frac{\partial U_{P^i,\Lambda_i}}{\partial P^i_j}.
\end{eqnarray}
In order to get the  useful expansions, we need to estimate each term on the right hand side of  \eqref{8} and \eqref{7} above.

\begin{lemma}\label{lm2}
There holds
\begin{eqnarray*}
\int_{\mathbb{R}^n}K\left(\frac x \lambda\right)W_m^{\frac{n+2s}{n-2s}}\frac{\partial U_{P^i,\Lambda_i}}{\partial \Lambda_i}&=&\int_{\mathbb{R}^n}K\left(\frac x \lambda\right)\sum_h U_{P^h,\Lambda_h}^{\frac{n+2s}{n-2s}}\frac{\partial U_{P^i,\Lambda_i}}{\partial \Lambda_i} \\
&&\;\;+\frac{n+2s}{n-2s}\int_{\mathbb{R}^n}K\left(\frac x \lambda\right)U_{P^i,\Lambda_i}^{\frac {4s}{n-2s}}\sum_{h\not= i} U_{P^h,\Lambda_h}\frac{\partial U_{P^i,\Lambda_i}}{\partial \Lambda_i}+O\left((\lambda l)^{-n}\right).
\end{eqnarray*}
\end{lemma}

\begin{proof}
We estimate the integration on different region. By the same method used in \eqref{43}, we have
\begin{equation}\label{57}
\int_{\cup_h(\Omega_h\cap B_h^c)}K(\frac x\lambda)W_m^{\frac{n+2s}{n-2s}}\big|\frac{\partial U_{P^i,\Lambda_i}}{\partial \Lambda_i}\big|=O\left((\lambda l)^{-n}\right).
\end{equation}

In the domain $\Omega_j\cap B_j$, where $j\not=i$, there holds $\hat{W}_{m,j}(y)\leq \sum_{h\not=j}\frac C{|X^j-X^h|}\leq \frac C{(\lambda l)^{{n-2s}}}\leq CU_{P^j,\Lambda_j}$. Taylor expansion yields
\begin{eqnarray*}
&&\int_{\Omega_j\cap B_j}K\left(\frac x\lambda\right)W_m^{\frac{n+2s}{n-2s}}\frac{\partial U_{P^i,\Lambda_i}}{\partial \Lambda_i} \\
&=&\int_{\Omega_j\cap B_j}K\left(\frac x\lambda\right)U_{P^j,\Lambda_j}^{\frac{n+2s}{n-2s}}\frac{\partial U_{P^i,\Lambda_i}}{\partial \Lambda_i}+O\left(\int_{\Omega_j\cap B_j}U_{P^j,\Lambda_j}^{\frac{4s}{n-2s}}\hat W_{m,j}\frac{\partial U_{P^i,\Lambda_i}}{\partial \Lambda_i}\right). \\
\end{eqnarray*}
For the error term, a direct computation yields
\begin{equation*}
\int_{\Omega_j\cap B_j}U_{P^j,\Lambda_j}^{\frac{4s}{n-2s}}\hat W_{m,j}\frac{\partial U_{P^i,\Lambda_i}}{\partial \Lambda_i}=O\left(\frac1{(\lambda l)^{2s}|X^i-X^j|^{n-2s}}\right).
\end{equation*}

\noindent\textbf{Claim:} For $j\not=i$, there holds
\begin{equation}\label{54}
\int_{\Omega_j\cap B_j}K\left(\frac x\lambda\right)U_{P^j,\Lambda_j}^{\frac{n+2s}{n-2s}}\frac{\partial U_{P^i,\Lambda_i}}{\partial \Lambda_i}=\int_{\mathbb{R}^n}K\left(\frac x\lambda\right)U_{P^j,\Lambda_j}^{\frac{n+2s}{n-2s}}\frac{\partial U_{P^i,\Lambda_i}}{\partial \Lambda_i}+O\left(\frac1{(\lambda l)^{2s}|X^i-X^j|^{n-2s}}\right).
\end{equation}
From direct computation
\begin{equation}\label{52}
\int_{\Omega_i\cap B_i}K(\frac x\lambda)U_{P^j,\Lambda_j}^{\frac{n+2s}{n-2s}}\frac{\partial U_{P^i, \Lambda_i}}{\partial \Lambda_i}=O\left(\frac{(\lambda l)^{2s}}{|X^i-X^j|^{n+2s}}\right).
\end{equation}
Using Lemma \ref{lm4}, we can obtain
\begin{eqnarray}\label{53}
\nonumber \sum_{h\not=i,j}\int_{\Omega_h\cap B_h}K(\frac x\lambda)U_{P^j,\Lambda_j}^{\frac{n+2s}{n-2s}}\left|\frac{\partial U_{P^i,\Lambda_i}}{\partial \Lambda_i}\right|&\leq&\sum_{h\not=i,j}\frac C{|X^i-X^j|^{n-2s}}\int_{\Omega_h\cap B_h}\frac 1{(1+|x-X^i|)^{n+2s}} \\
&\leq& \frac C{|X^i-X^j|^{n-2s}}\sum_{h\not=i}\frac{(\lambda l)^n}{|X^h-X^i|^{n+2s}} \\ \nonumber
&\leq& \frac C{(\lambda l)^{2s}|X^i-X^j|^{n-2s}},
\end{eqnarray}
and
\begin{equation}\label{51}
\int_{\cup_h(\Omega_h\cap B_h^c)}K(\frac x\lambda)U_{P^j,\Lambda_j}^{\frac{n+2s}{n-2s}}\big|\frac{\partial U_{P^i,\Lambda_i}}{\partial \Lambda_i}\big|=O\left(\frac1{(\lambda l)^{2s}|X^i-X^j|^{n-2s}}\right).
\end{equation}
From \eqref{52}, \eqref{53} and \eqref{51}, we know the Claim is true.

Hence for $j\not=i$,
\begin{equation}\label{58}
\int_{\Omega_j\cap B_j}K\left(\frac x\lambda\right)W_m^{\frac{n+2s}{n-2s}}\frac{\partial U_{P^i,\Lambda_i}}{\partial \Lambda_i}=\int_{\mathbb{R}^n}K\left(\frac x\lambda\right)U_{P^j,\Lambda_j}^{\frac{n+2s}{n-2s}}\frac{\partial U_{P^i,\Lambda_i}}{\partial \Lambda_i}+ O\left(\frac1{(\lambda l)^{2s}|X^i-X^j|^{n-2s}}\right).
\end{equation}

Now we estimate the integration on $\Omega_i\cap B_i$. By Taylor expansion,
\begin{eqnarray}\label{55}
&&\int_{\Omega_i\cap B_i}K\left(\frac x\lambda\right)W_m^{\frac{n+2s}{n-2s}}\frac{\partial U_{P^i,\Lambda_i}}{\partial \Lambda_i}  \nonumber \\
&&=\int_{\Omega_i\cap B_i}K\left(\frac x\lambda\right)U_{P^i, \Lambda_i}^{\frac{n+2s}{n-2s}}\frac{\partial U_{P^i,\Lambda_i}}{\partial \Lambda_i}+\frac{n+2s}{n-2s}\int_{\Omega_i\cap B_i}K\left(\frac x\lambda\right)U_{P^i,\Lambda_i}^{\frac{4s}{n-2s}}\sum_{h\not=i} U_{P^h,\Lambda_h}\frac{\partial U_{P^i,\Lambda_i}}{\partial \Lambda_i} \nonumber\\
&&\hspace{2,cm}+O\left(\int_{\Omega_i\cap B_i}K(\frac x\lambda)U_{P^i,\Lambda_i}^{\frac{4s}{n-2s}}\hat W_{m,i}^2\right) \nonumber\\
&&=\int_{\Omega_i\cap B_i}K\left(\frac x\lambda\right)U_{P^i, \Lambda_i}^{\frac{n+2s}{n-2s}}\frac{\partial U_{P^i,\Lambda_i}}{\partial \Lambda_i}+\frac{n+2s}{n-2s}\int_{\Omega_i\cap B_i}K\left(\frac x\lambda\right)U_{P^i,\Lambda_i}^{\frac{4s}{n-2s}}\sum_{h\not=i} U_{P^h,\Lambda_h}\frac{\partial U_{P^i,\Lambda_i}}{\partial \Lambda_i} \nonumber \\
&&\hspace{2,cm}+O\left((\lambda l)^{-n}\right).
\end{eqnarray}
Since in the domain $\Omega_i^c\cup B_i^c$, we have $|y-X^i|\geq \min\{\lambda l,\min_{j\not=i}\frac 12|X^i-X^j|\}\geq \frac12\lambda l$. Then
\begin{equation}\label{56}
\int_{\Omega_i\cap B_i}K\left(\frac x\lambda\right)U_{P^i, \Lambda_i}^{\frac{n+2s}{n-2s}}\frac{\partial U_{P^i,\Lambda_i}}{\partial \Lambda_i}=\int_{\mathbb{R}^n}K\left(\frac x\lambda\right)U_{P^i, \Lambda_i}^{\frac{n+2s}{n-2s}}\frac{\partial U_{P^i,\Lambda_i}}{\partial \Lambda_i}+O\left((\lambda l)^{-n}\right).
\end{equation}
By a similar method used in the proof of \eqref{54}, we get
\begin{eqnarray}\label{60}
\int_{\Omega_i\cap B_i}K\left(\frac x\lambda\right)U_{P^i,\Lambda_i}^{\frac{4s}{n-2s}}U_{P^j,\Lambda_j}\frac{\partial U_{P^i,\Lambda_i}}{\partial \Lambda_i}&=&\int_{\mathbb{R}^n}K\left(\frac x\lambda\right)U_{P^i,\Lambda_i}^{\frac{4s}{n-2s}}U_{P^j,\Lambda_j}\frac{\partial U_{P^i,\Lambda_i}}{\partial \Lambda_i} \\
&&+O\left(\frac1{(\lambda l)^{2s}|X^i-X^j|^{n-2s}}\right).\nonumber
\end{eqnarray}
Substituting \eqref{56} and \eqref{60} into \eqref{55}, we have
\begin{eqnarray}\label{59}
\nonumber&&\int_{\Omega_i\cap B_i}K\left(\frac x\lambda\right)W_m^{\frac{n+2s}{n-2s}}\frac{\partial U_{P^i,\Lambda_i}}{\partial \Lambda_i} \\
&=&\int_{\mathbb{R}^n}K\left(\frac x\lambda\right)U_{P^i, \Lambda_i}^{\frac{n+2s}{n-2s}}\frac{\partial U_{P^i,\Lambda_i}}{\partial \Lambda_i}+\frac{n+2s}{n-2s}\int_{\mathbb{R}^n}K\left(\frac x\lambda\right)U_{P^i,\Lambda_i}^{\frac{4s}{n-2s}}\sum_{h\not=i} U_{P^h,\Lambda_h}\frac{\partial U_{P^i,\Lambda_i}}{\partial \Lambda_i} \\ \nonumber
&&+O\left((\lambda l)^{-n}\right).
\end{eqnarray}

Now Lemma \ref{lm2} follows from the estimate \eqref{57}, \eqref{58} and \eqref{59}.

\end{proof}

\begin{lemma}
For $h\not=i$, there holds
\begin{equation}\label{10}
\int_{\mathbb{R}^n}K\left(\frac x\lambda\right)U_{P^h,\Lambda_h}^{\frac{n+2s}{n-2s}}\frac{\partial U_{P^i,\Lambda_i}}{\partial \Lambda_i}=\int_{\mathbb{R}^n}U_{P^h,\Lambda_h}^{\frac{n+2s}{n-2s}}\frac{\partial U_{P^i,\Lambda_i}}{\partial \Lambda_i}+O\left(\frac1{\lambda^{2s}|X^i-X^h|^{n-2s}}\right),
\end{equation}
and
\begin{equation}\label{9}
\frac{n+2s}{n-2s}\int_{\mathbb{R}^n}K\left(\frac x\lambda\right)U_{P^i,\Lambda_i}^{\frac{4s}{n-2s}}U_{P^h,\Lambda_h}\frac{\partial U_{P^i,\Lambda_i}}{\partial \Lambda_i}=\int_{\mathbb{R}^n}U_{P^h,\Lambda_h}^{\frac{n+2s}{n-2s}}\frac{\partial U_{P^i,\Lambda_i}}{\partial \Lambda_i}+O\left(\frac1{\lambda^{2s}|X^i-X^h|^{n-2s}}\right).
\end{equation}
\end{lemma}
\begin{proof}
Notice the fact
\begin{eqnarray*}
\int_{\mathbb{R}^n}U_{P^h,\Lambda_h}^{\frac{n+2s}{n-2s}}\frac{\partial U_{P^i,\Lambda_i}}{\partial \Lambda_i}&=&\int_{\mathbb{R}^n}(-\Delta)^sU_{P^h,\Lambda_h} \frac{\partial U_{P^i,\Lambda_i}}{\partial \Lambda_i}  \\
&=& \int_{\mathbb{R}^n} U_{P^h,\Lambda_h} (-\Delta)^s \frac{\partial U_{P^i,\Lambda_i}}{\partial \Lambda_i} \\
&=& \frac{n+2s}{n-2s} \int_{\mathbb{R}^n} U_{P^h,\Lambda_h}U_{P^i,\Lambda_i}^{\frac{4s}{n-2s}} \frac{\partial U_{P^i,\Lambda_i}}{\partial \Lambda_i}.
\end{eqnarray*}
So the proof of \eqref{9} and \eqref{10} are identical. We only give a proof of \eqref{10}.

Choose $\delta$ to be a fixed constant some enough. Since $n>4s>n+2s-\beta$, the condition ($H_3$) implies

\begin{eqnarray}\label{69}
\nonumber&&\left|\int_{B_{\delta \lambda}(X^h)}\left|K\left(\frac x \lambda\right)-1\right|U_{P^h,\Lambda_h}^{\frac{n+2s}{n-2s}}\frac{\partial U_{P^i,\Lambda_i}}{\partial \Lambda_i}\right| \\ \nonumber
&\leq& C\int_{B_{\delta\lambda}(X^h)}\frac{|x-X^h|^\beta}{\lambda^\beta}\frac1{(1+|x-P^h|)^{n+2s}}\frac1{(1+ |x-P^i|)^{n-2s}} \\
&\leq& \frac C{\lambda^\beta|X^i-X^h|^{n-2s}}\int_{B_{\delta\lambda}(X^h)}\frac{|x-X^h|^\beta}{(1+|x-X^h|)^{n+2s}}\leq
\frac C{\lambda^{2s}|X^i-X^h|^{n-2s}}.
\end{eqnarray}
A direct calculation yields
\begin{equation}\label{70}
\left|\int_{B_{\delta \lambda}(X^i)}\left|K\left(\frac x \lambda\right)-1\right|U_{P^h,\Lambda_h}^{\frac{n+2s}{n-2s}}\frac{\partial U_{P^i,\Lambda_i}}{\partial \Lambda_i}\right|\leq \frac{C\lambda^{2s}}{|X^i-X^h|^{n+2s}}\leq \frac C{\lambda^{2s}|X^i-X^h|^{n-2s}}.
\end{equation}
Using Lemma \ref{lm4}, we have
\begin{eqnarray}\label{71}
\nonumber&&\left|\int_{B^c_{\delta\lambda}(X^h)\cap B^c_{\delta\lambda}(X^i)}\left|K\left(\frac x \lambda\right)-1\right|U_{P^h,\Lambda_h}^{\frac{n+2s}{n-2s}}\frac{\partial U_{P^i,\Lambda_i}}{\partial \Lambda_i}\right| \\ \nonumber
&\leq& C\int_{B^c_{\delta\lambda}(X^h)\cap B^c_{\delta\lambda}(X^i)}\frac1{(1+|x-X^h|)^{n+2s}}\frac1{ (1+|x-X^i|)^{n-2s}} \\
&\leq& \frac C{|X^h-X^i|^{n-2s}}\int_{B^c_{\delta\lambda}(X^h)}\frac1{(1+|x-X^h|)^{n+2s}}\leq \frac C{\lambda^{2s}|X^i-X^h|^{n-2s}}.
\end{eqnarray}
Hence \eqref{10} follows from \eqref{69}, \eqref{70} and \eqref{71}.

\end{proof}

\begin{lemma}\label{lm1}
We have
\begin{eqnarray}\label{6}
\nonumber\int_{\mathbb{R}^n}K\left(\frac x \lambda\right) U_{P^i,\Lambda_i}^{\frac{n+2s}{n-2s}}\frac{\partial U_{P^i,\Lambda_i}}{\partial \Lambda_i}&=&-\frac{n-2s}{2n}\frac{\beta C_0(n,s)^{\frac{2n}{n-2s}}(\sum_h a_h)}{\Lambda_i^{\beta+1}\lambda^\beta}\int_{\mathbb{R}^n}\frac{|x_1|^\beta}{(1+|x|^2)^n} \\
&&\;\;\;\;+O(\frac{|P^i-X^i|^{\min\{2,\beta-1\}}}{\lambda^\beta})+o\left(\lambda^{-\beta}\right),
\end{eqnarray}
and
\begin{eqnarray}\label{5}
\nonumber\int_{\mathbb{R}^n}K\left(\frac x \lambda\right)U_{P^i,\Lambda_i}^{\frac{n+2s}{n-2s}}\frac{\partial U_{P^i,\Lambda_i}}{\partial P^i_j}&=&\frac{(n-2s)C_0(n,s)^{\frac{2n}{n-2s}}\beta a_j}{\Lambda_i^{\beta-2} \lambda^\beta} \int_{\mathbb{R}^n}\frac{|x_1|^\beta}{(1+|x|^2)^{n+1}}(P^i_j-X^i_j) \\
&&\;\;\;\;+O\left(\frac {|P^i-X^i|^2} {\lambda^\beta}\right)+o\left(\lambda^{-\beta}\right),
\end{eqnarray}
where $i=1,\dots,(m+1)^k$ and $j=1,\dots,n$.
\end{lemma}
\begin{proof}
The two formulas  follows from some standard calculations, see \cite[Lemma A.9, Lemma A.10]{LiYY2016} for details.

\end{proof}

\begin{lemma}\label{lm3}
For $h\not=i$, we have
\begin{equation*}
\int_{\mathbb{R}^n} U_{P^h,\Lambda_h}^{\frac{n+2s}{n-2s}}\frac{\partial U_{P^i,\Lambda_i}}{\partial \Lambda_i}=c_0\frac{\partial\varepsilon_{ih}}{\partial \Lambda_i}+\frac1{\Lambda_i}O(\varepsilon_{hi}^{\frac{n}{n-2s}}\log \varepsilon_{hi}),
\end{equation*}
where $c_0=C_0(n,s)^{\frac{2n}{n-2s}}\int_{\mathbb{R}^n}\frac1{(1+|y|^2)^{\frac{n+2s}2}}$ and $\varepsilon_{ih} =\left(\frac1{\frac{\Lambda_i}{\Lambda_h}+\frac{\Lambda_h}{\Lambda_i}+\Lambda_i\Lambda_h |P^i-P^h|^2}\right)^\frac{n-2s}2$.
\end{lemma}
\begin{proof}
The proof of this lemma is rather standard. We  refer to \cite{Bahri1989} and \cite{Chen2016} for ideas.

\end{proof}

\begin{proposition}\label{prop3}
It holds that
\begin{equation*}
\frac{\partial I}{\partial \Lambda_i}(W_m)=-\frac{c_1}{\Lambda_i^{\beta+1}\lambda^\beta}+\sum_{h\not=i}\frac{c_2}{ \Lambda_i(\Lambda_i\Lambda_h)^{\frac{n-2s}2}|X^i-X^h|^{n-2s}}+O\left(\frac{|P^i-X^i|^{\min\{2,\beta-1\}}}{ \lambda^\beta}\right)+o(\lambda^{-\beta}),
\end{equation*}
where $c_1=\frac{(n-2s)\beta C_0(n,s)^{\frac {2n}{n-2s}}(-\sum_h a_h)}{2n}\int_{\mathbb{R}^n}\frac{|x_1|^\beta}{ (1+|x|^2)^n}>0$ and $c_2=\frac{n-2s}2 C_0(n,s)^{\frac{2n}{n-2s}}\int_{\mathbb{R}^n}\frac1{(1+|y|^2)^{\frac{n+2s}2}}$.
\end{proposition}

\begin{proof}
This proposition is a consequence of Lemma \ref{lm2}, \eqref{10}, \eqref{9}, \eqref{6}, Lemma \ref{lm3} and the definition of $\lambda$. We need to remind that
\begin{equation*}
\frac{\partial\varepsilon_{ih}}{\partial \Lambda_i}= -\frac{n-2s}{2\Lambda_i(\Lambda_i\Lambda_h)^{\frac{n-2s}2}|X^i-X^h|^{n-2s}}+O(\frac1{|X^i-X^h|^{n-2s+1}}),
\end{equation*}
which is directly from   $P_h\in B_{\frac12}(X^h)$ and the definition of $\{X^h\}_{h=1}^{(m+1)^k}$.
\end{proof}

\begin{proposition}\label{prop2}
We have
\begin{equation*}
\frac {\partial I} {\partial P^i_j}(W_m)=-\frac{c_3a_j}{\Lambda_i^{\beta-2}\lambda^\beta}(P^i_j-X^i_j) +O\left(\frac{|P^i-X^i|^2}{\lambda^\beta}\right)+o(\lambda^{-\beta}),
\end{equation*}
where $c_3=(n-2s)C_0(n,s)^{\frac {2n}{n-2s}}\beta\int \frac {|x_1|^\beta}{(1+|x|^2)^{n+1}}$.
\end{proposition}

\begin{proof}
We need to estimate each term on the right hand side of the equality \eqref{7}. By simple calculation, we have
\begin{equation*}
\left|W_m^{\frac{n+2s}{n-2s}}-U_{P^i,\Lambda_i}^{\frac{n+2s}{n-2s}}-\frac{n+2s}{n-2s}U_{P^i, \Lambda_i}^{\frac{4s}{n-2s}}\sum_{h\not=i} U_{P^h,\Lambda_h}\right|\leq\left\{\begin{array}{ll}\left(\sum_{h\not=i} U_{P^h,\Lambda_h}\right)^{\frac{n+2s}{n-2s}}, &\mbox{if $U_{P^i,\Lambda_i}\leq\sum_{h\not=i} U_{P^h,\Lambda_h}$,} \\ U_{P^i,\Lambda_i}^{\frac{6s-n}{n-2s}}\left(\sum_{h\not=i}U_{P^h,\Lambda_h} \right)^2, &\mbox{otherwise.}\end{array}\right.
\end{equation*}
Then we get
\begin{eqnarray}\label{11}
\nonumber&&\int_{\mathbb{R}^n}K\left(\frac x \lambda\right)W_m^{\frac{n+2s}{n-2s}}\frac{\partial U_{P^i,\Lambda_i}}{\partial P^i_j} \\
&=&\int_{\mathbb{R}^n}K\left(\frac x \lambda\right)U_{P^i,\Lambda_i}^{\frac{n+2s}{n-2s}}\frac{\partial U_{P^i,\Lambda_i}}{\partial P^i_j}+ \frac{n+2s}{n-2s}\int_{\mathbb{R}^n}K\left(\frac x \lambda\right)U_{P^i,\Lambda_i}^{\frac{4s}{n-2s}}\sum_{h\not=i}U_{P^h,\Lambda_h}\frac{\partial U_{P^i,\Lambda_i}}{\partial P^i_j} \\
\nonumber&&+O\left(\int_{\mathbb{R}^n}(\sum_{h\not=i} U_{P^h,\Lambda_h})^{\frac{n+2s}{n-2s}}\frac{\partial U_{P^i,\Lambda_i}}{\partial P^i_j}\right)+O\left(\int_{U_{P^i,\Lambda_i}>\sum_{h\not=i} U_{P^h,\Lambda_h}}U_{P^i,\Lambda_i}^{\frac{6s-n}{n-2s}}(\sum_{h\not=i}U_{P^h,\Lambda_h} )^2\frac{\partial U_{P^i,\Lambda_i}}{\partial P^i_j}\right).
\end{eqnarray}
We first estimate the error terms above. Since $n>2s+2>4s$, we get $(n-s)\frac{n-2s}{n+2s}>\frac{n-2s}2>k$. From Lemma \ref{lm4} we have
\begin{eqnarray}\label{12}
\nonumber\int_{\mathbb{R}^n}(\sum_{h\not=i} U_{P^h,\Lambda_h})^{\frac{n+2s}{n-2s}}\big|\frac{\partial U_{P^i,\Lambda_i}}{\partial P^i_j}\big|&\leq& C\int_{\mathbb{R}^n}\left(\sum_{h\not=i}\frac1{(1+|y-X^h|)^{n-2s} }\right)^{\frac{n+2s}{n-2s}}\frac1{(1+|y-X^i|)^{n-2s+1}} \\
\nonumber&\leq& C\int_{\mathbb{R}^n}\left(\sum_{h\not=i}\frac1{(1+|y-X^h|)^{n-2s}}\frac1{(1+|y-X^i|)^{(n-s)\frac{n-2s}{n+2s}}} \right)^{\frac{n+2s}{n-2s}} \\
\nonumber&\leq& \left(\sum_{h\not=i}\frac1{|X^h-X^i|^{(n-s)\frac{n-2s}{n+2s}}}\right)^{\frac{n+2s}{n-2s}}\int_{ \mathbb{R}^n}\frac1{(1+|y-X^i|)^{n+2s}} \\
&\leq& C(\lambda l)^{-(n-s)}.
\end{eqnarray}
The similar argument yields
\begin{eqnarray}\label{13}
\nonumber\int_{U_{P^i,\Lambda_i}>\sum_{h\not=i} U_{P^h,\Lambda_h}}U_{P^i,\Lambda_i}^{\frac{6s-n}{n-2s}}(\sum_{h\not=i}U_{P^h,\Lambda_h} )^2\left|\frac{\partial U_{P^i,\Lambda_i}}{\partial P^i_j}\right|&\leq& \int_{\mathbb{R}^n} (\sum_{h\not=i}U_{P^h,\Lambda_h} )^{\frac{n}{n-2s}}U_{P^i,\Lambda_i}^{\frac {2s}{n-2s}}\big|\frac{\partial U_{P^i,\Lambda_i}}{\partial P^i_j}\big| \\
&\leq& C(\lambda l)^{-n}.
\end{eqnarray}

For $h\not=i$, we  see that
\begin{eqnarray}\label{15}
\nonumber&&\frac{n+2s}{n-2s}\int_{\mathbb{R}^n}K\left(\frac x \lambda\right)U_{P^i,\Lambda_i}^{\frac{4s}{n-2s}}U_{P^h,\Lambda_h}\frac{\partial U_{P^i,\Lambda_i}}{\partial P^i_j} \\
\nonumber&=&\frac{\partial}{\partial P^i_j}\int_{\mathbb{R}^n}K\left(\frac x \lambda\right)U_{P^i,\Lambda_i}^{\frac{n+2s}{n-2s}}U_{P^h,\Lambda_h} \\
\nonumber&=&\frac1\lambda\int_{\mathbb{R}^n}\frac{\partial K}{\partial t_j}\left(\frac{x+P^i}{\lambda}\right)U_{0,\Lambda_i}^{\frac{n+2s}{n-2s}}U_{P^h-P^i,\Lambda_h}- \int_{\mathbb{R}^n}K\left(\frac{x+P^i}{\lambda}\right)U_{0,\Lambda_i}^{\frac{n+2s}{n-2s}}\frac{\partial U_{P^h-P^i,\Lambda_h}}{\partial P^i_j} \\
&=&O(\lambda^{-1}\frac1{|X^i-X^h|^{n-2s}})+O(\frac1{|X^i-X^h|^{n-s}}).
\end{eqnarray}
The first part of \eqref{7} can be estimated as
\begin{eqnarray}\label{14}
\nonumber\int_{\mathbb{R}^n}\sum_{h\not=i} U_{P^h,\Lambda_h}^{\frac{n+2s}{n-2s}}\big|\frac{\partial U_{P^i,\Lambda_i}}{\partial P^i_j}\big|
&\leq&\sum_{h\not=i}\int_{\mathbb{R}^n}\frac C{(1+|x-X^h|)^{n+2s}}\frac1{(1+|x-X^i|)^{n-2s+1}} \\
\nonumber&\leq& \sum_{h\not=i}\frac C{|X^h-X^i|^{n-s}}\int_{\mathbb{R}^n}\frac1{(1+|y-X^h|)^{n+s+1}} \\
&\leq& C(\lambda l)^{-(n-s)}.
\end{eqnarray}
Then the expansion of $\frac{\partial I}{\partial P^i_j}(W_m)$ follows from \eqref{5}, \eqref{11}, \eqref{12}, \eqref{13}, \eqref{15} and \eqref{14}.

\end{proof}

\bibliographystyle{amsplain}

\end{document}